\documentclass[a4paper,12pt]{amsart}

\usepackage{amsmath,amsfonts,amsthm,amssymb,amscd,enumerate,esint,vmargin}
\usepackage[bookmarks, colorlinks=true, pdftitle={Interpolation and embeddings of weighted tent spaces}, pdfauthor={Alex Amenta}]{hyperref}

\newcommand{\CC}{\mathbb{C}}
\newcommand{\NN}{\mathbb{N}}
\newcommand{\RR}{\mathbb{R}}
\newcommand{\ZZ}{\mathbb{Z}}
\newcommand{\ga}{\alpha}

\newcommand{\gd}{\delta}
\renewcommand{\ge}{\varepsilon}

\newcommand{\gf}{\varphi}

\newcommand{\gh}{\eta}

\newcommand{\gk}{\kappa}
\newcommand{\gl}{\lambda}
\newcommand{\gm}{\mu}

\newcommand{\gp}{\pi}
\newcommand{\gq}{\theta}
\newcommand{\gr}{\rho}

\newcommand{\gs}{\sigma}
\newcommand{\gt}{\tau}

\newcommand{\gx}{\xi}
\newcommand{\gy}{\psi}

\newcommand{\gG}{\Gamma}

\newcommand{\gO}{\Omega}

\newcommand{\mc}{\mathcal}
\newcommand{\mb}{\mathbf}
\newcommand{\map}[3]{#1 \colon #2 \rightarrow #3}
\newcommand{\norm}[1]{\left|\left|#1\right|\right|}
\newcommand{\diam}{\operatorname{diam}}
\newcommand{\supp}{\operatorname{supp}}
\newcommand{\dist}{\operatorname{dist}}
\DeclareMathOperator*{\esssup}{ess\,sup}
\newcommand{\sm}{\setminus}
\newcommand{\td}{\tilde}
\newcommand{\wtd}{\widetilde}

\theoremstyle{theorem}
\newtheorem{thm}{Theorem}[section]
\theoremstyle{definition}
\newtheorem{dfn}[thm]{Definition}
\theoremstyle{remark}
\newtheorem{rmk}[thm]{Remark}
\theoremstyle{plain}
\newtheorem{lem}[thm]{Lemma}
\theoremstyle{plain}
\newtheorem{cor}[thm]{Corollary}
\newtheorem{prop}[thm]{Proposition}
\theoremstyle{definition}

\newcommand\blfootnote[1]{%
  \begingroup
  \renewcommand\thefootnote{}\footnote{#1}%
  \addtocounter{footnote}{-1}%
  \endgroup
}

\def\barint_#1^#2{\mathchoice
            {\mathop{\vrule width 6pt
height 3 pt depth -2.5pt
                    \kern -8.8pt
\int_{#1}^{#2} \kern 0pt}}%
            {\mathop{\vrule width 5pt height
3 pt depth -2.6pt
                    \kern -6.5pt
\intop_{#1}^{#2} \kern -4pt}\nolimits_{#1}}%
            {\mathop{\vrule width 5pt height
3 pt depth -2.6pt
                    \kern -6pt
\intop_{#1}^{#2} \kern -4pt}\nolimits_{#1}}%
            {\mathop{\vrule width 5pt height
3 pt depth -2.6pt
          \kern -6pt \intop \kern -4pt}\nolimits_{#1}}}
          
\def\bariint_#1^#2{\mathchoice
            {\mathop{\vrule width 10pt
height 3 pt depth -2.5pt
                    \kern -12.8pt
\intop \kern -10pt\int_{#1}^{#2} \kern 0pt}}%
            {\mathop{\vrule width 9pt height
3 pt depth -2.6pt
                    \kern -10.5pt
\intop \kern -10pt\int_{#1}^{#2} \kern -4pt}}%
            {\mathop{\vrule width 9pt height
3 pt depth -2.6pt
                    \kern -10pt
\intop \kern -10pt\int_{#1}^{#2} \kern -4pt}}%
            {\mathop{\vrule width 9pt height
3 pt depth -2.6pt
          \kern -10pt \int_{#1}^{#2} \kern -10pt\intop \kern -4pt}
      }}
      
\renewcommand{\iint}{\int \kern -10pt\int}

\title{Interpolation and embeddings of weighted tent spaces}
\author{Alex Amenta}
\address{Laboratoire de MathŽmatiques d'Orsay, Univ. Paris-Sud, CNRS, Universit\'e
Paris-Saclay, 91405 Orsay, France} 
\curraddr{Delft Institute of Applied Mathematics, Delft University of Technology, P.O. Box
5031, 2628 CD Delft, The Netherlands}

\begin{document}
\maketitle

\blfootnote{2010 \emph{Mathematics Subject Classification}. 42B35 (Primary); 46E30 (Secondary). \\
\emph{Keywords.} Weighted tent spaces, complex interpolation, real interpolation, Hardy--Littlewood--Sobolev embeddings.}

\begin{abstract}
	Given a metric measure space $X$, we consider a scale of function spaces $T^{p,q}_s(X)$, called the \emph{weighted tent space scale}.
	This is an extension of the tent space scale of Coifman, Meyer, and Stein.
	Under various geometric assumptions on $X$ we identify some associated interpolation spaces, in particular certain real interpolation spaces.
	These are identified with a new scale of function spaces, which we call \emph{$Z$-spaces}, that have recently appeared in the work of Barton and Mayboroda on elliptic boundary value problems with boundary data in Besov spaces.
	We also prove Hardy--Littlewood--Sobolev-type embeddings between weighted tent spaces.
\end{abstract}

\tableofcontents

The \emph{tent spaces}, denoted $T^{p,q}$, are a scale of function spaces first introduced by Coifman, Meyer, and Stein \cite{CMS83, CMS85} which have had many applications in harmonic analysis and partial differential equations.
In some of these applications `weighted' tent spaces have been used implicitly.
These spaces, which we denote by $T^{p,q}_s$, seem not to have been considered as forming a scale of function spaces in their own right until the work of Hofmann, Mayboroda, and McIntosh \cite[\textsection 8.3]{HMM11}, in which factorisation and complex interpolation theorems are obtained for these spaces.

In this article we further explore the weighted tent space scale.
In the interests of generality, we consider weighted tent spaces $T^{p,q}_s(X)$ associated with a metric measure space $X$, although our theorems are new even in the classical case where $X = \RR^n$ equipped with the Lebesgue measure.
Under sufficient geometric assumptions on $X$ (ranging from the doubling condition to the assumption that $X = \RR^n$), we uncover two previously unknown novelties of the weighted tent space scale.

First, we identify some real interpolation spaces between $T^{p_0,q}_{s_0}$ and $T^{p_1,q}_{s_1}$ whenever $s_0 \neq s_1$.
In Theorem \ref{rint} we prove that
\begin{equation}
	(T^{p_0,q}_{s_0}, T^{p_1,q}_{s_1})_{\gq,p_\gq} = Z^{p_\gq, q}_{s_\gq}
\end{equation}
for appropriately defined parameters, where the scale of `$Z$-spaces' is defined in Definition \ref{Zdfn}.
We require $p_0,p_1,q > 1$ in this identification, but in Theorem \ref{rint2} we show that in the Euclidean setting the result holds for all $p_0,p_1 > 0$ and $q \geq 1$.
In the Euclidean setting, $Z$-spaces have appeared previously in the work of Barton and Mayboroda \cite{BM16}.
In their notation we have $Z^{p,q}_s(\RR^n) = L(p,ns + 1,q)$.
Barton and Mayboroda show that these function spaces are useful in the study of elliptic boundary value problems with boundary data in Besov spaces.
The connection with weighted tent spaces shown here is new.

Second, we have continuous embeddings
\begin{equation*}
	T^{p_0,q}_{s_0} \hookrightarrow T^{p_1,q}_{s_1}
\end{equation*}
whenever the parameters satisfy the relation
\begin{equation}\label{HLS}
	s_1 - s_0 = \frac{1}{p_1} - \frac{1}{p_0}.
\end{equation}
This is Theorem \ref{embthm1}.
Thus a kind of Hardy--Littlewood--Sobolev embedding theorem holds for the weighted tent space scale, and by analogy we are justified in referring to the parameter $s$ in $T^{p,q}_s$ as a \emph{regularity} parameter.

We also identify complex interpolation spaces between weighted tent spaces in the Banach range.
This result is already well-known in the Euclidean setting, and its proof does not involve any fundamentally new arguments, but we include it here for completeness.

These results in this paper play a crucial role in recent work of the author and Pascal Auscher \cite{AA16}, in which we use weighted tent spaces and $Z$-spaces to construct abstract homogeneous Hardy--Sobolev and Besov spaces associated with elliptic differential operators with rough coefficients, extending the abstract Hardy space techniques initiated independently by Auscher, McIntosh, and Russ \cite{AMR08} and Hofmann and Mayboroda \cite{HM09}.

\subsection*{Notation}

Given a measure space $(X,\gm)$, we write $L^0(X)$ for the set of $\gm$-measurable functions with values in the extended complex numbers $\CC \cup \{\pm \infty, \pm i\infty\}$.
As usual, by a `measurable function', we actually mean an equivalence class of measurable functions which are equal except possibly on a set of measure zero.
We will say that a function $f \in L^0(X)$ is essentially supported in a subset $E \subset X$ if ${\gm\{x \in X \sm E : f(x) \neq 0\} = 0}$.

A \emph{quasi-Banach space} is a complete quasi-normed vector space; see for example \cite[\textsection 2]{nK03} for further information.
If $B$ is a quasi-Banach space, we will write the quasi-norm of $B$ as either $\norm{\cdot}_B$ or $\norm{\cdot \mid B}$, according to typographical needs.

For $1 \leq p \leq \infty$, we let $p^\prime$ denote the H\"older conjugate of $p$, which is defined by the relation
\begin{equation*}
	1 = \frac{1}{p} + \frac{1}{p^\prime},
\end{equation*}
with $1/\infty := 0$.
For $0 < p,q \leq \infty$, we define the number
\begin{equation*}
	\gd_{p,q} := \frac{1}{q} - \frac{1}{p},
\end{equation*}
again with $1/\infty := 0$.
This shorthand will be used often throughout this article.
We will frequently use the the identities
\begin{align*}
	\gd_{p,q} + \gd_{q,r} &= \gd_{p,r}, \\
	\gd_{p,q} &= \gd_{q^\prime, p^\prime}, \\
	1/q &= \gd_{\infty,q} = \gd_{q^\prime, 1}.
\end{align*}

As is now standard in harmonic analysis, we write $a \lesssim b$ to mean that $a \leq Cb$ for some unimportant constant $C \geq 1$ which will generally change from line to line.
We also write $a \lesssim_{c_1,c_2,\ldots} b$ to mean that $a \leq C(c_1,c_2,\ldots) b$.

\subsection*{Acknowledgements}
I thank Pierre Portal and Pascal Auscher for their comments, suggestions, corrections, and support.
I also thank Yi Huang, Gennady Uraltsev, and Christian Zillinger for interesting discussions regarding technical aspects of this work.
The referee provided some valuable corrections and suggestions, so I thank them too.

Some of this work was carried out during the Junior Trimester Program on Optimal Transportation at the Hausdorff Research Institute for Mathematics in Bonn.
I thank the Hausdorff Institute for their support, as well as the participants of the trimester seminar for their comments.
I also thank Lashi Bandara and the \emph{Pr\'efecture de Bobigny} for indirectly influencing my participation in this program.

Finally, I acknowledge financial support from the Australian Research Council Discovery Grant DP120103692, and also from the ANR project ``Harmonic analysis at its boundaries'' ANR-12-BS01-0013-01.

\section{Preliminaries}

\subsection{Metric measure spaces}

A \emph{metric measure space} is a triple $(X,d,\gm)$, where $(X,d)$ is a nonempty metric space and $\gm$ is a Borel measure on $X$.
For every $x \in X$ and $r > 0$, we write $B(x,r) := \{y \in X : d(x,y) < r\}$ for the ball of radius $r$, and we also write $V(x,r) := \gm(B(x,r))$ for the volume of this ball.
The \emph{generalised half-space} associated with $X$ is the set $X^+ := X \times \RR_+$, equipped with the product topology and the product measure $d\gm(y) \, dt/t$.

We say that $(X,d,\gm)$ is \emph{nondegenerate} if
\begin{equation}
	0 < V(x,r) < \infty \quad \text{for all $x \in X$ and $r > 0$.}
\end{equation}
This immediately implies that the measure space $(X,\gm)$ is $\gs$-finite, as $X$ may be written as an increasing sequence of balls
\begin{equation}\label{exhaustion}
	X = \bigcup_{n \in \NN} B(x_0,n)
\end{equation}
for any point $x_0 \in X$.
Nondegeneracy also implies that the metric space $(X,d)$ is separable \cite[Proposition 1.6]{BB11}.
To rule out pathological behaviour (which is not particularly interesting from the viewpoint of tent spaces), we will always assume nondegeneracy.

Generally we will need to make further geometric assumptions on $(X,d,\gm)$.
In this article, the following two conditions will be used at various points.
We say that $(X,d,\gm)$ is \emph{doubling} if there exists a constant $C \geq 1$ such that
\begin{equation*}
	V(x,2r) \leq CV(x,r) \quad \text{for all $(x,r) \in X^+$}.
\end{equation*}
A consequence of the doubling condition is that there exists a minimal number $n \geq 0$, called the \emph{doubling dimension} of $X$, and a constant $C \geq 1$ such that
\begin{equation*}
	V(x,R) \leq C (R/r)^n V(x,r)
\end{equation*}
for all $x \in X$ and $0 < r \leq R < \infty$.

For $n > 0$, we say that $(X,d,\gm)$ is \emph{AD-regular of dimension $n$} if there exists a constant $C \geq 1$ such that
\begin{equation}\label{AD}
	C^{-1} r^n \leq V(x,r) \leq Cr^n
\end{equation}
for all $x \in X$ and all $r < \diam(X)$.
One can show that AD-regularity (of some dimension) implies doubling.
Note that if $X$ is unbounded and AD-regular of dimension $n$, then \eqref{AD} holds for all $x \in X$ and all $r > 0$.

\subsection{Unweighted tent spaces}

Throughout this section we suppose that $(X,d,\gm)$ is a nondegenerate metric measure space.
We will not assume any further geometric conditions on $X$ without explicit mention.
All of the results here are known, at least in some form.
We provide statements for ease of reference and some proofs for completeness.

For $x \in X$ we define the \emph{cone with vertex $x$} by
\begin{equation*}
	\gG(x) := \{(y,t) \in X^+ : y \in B(x,t)\},
\end{equation*}
and for each ball $B \subset X$ we define the \emph{tent with base $B$} by
\begin{equation*}
	T(B) := X^+ \sm \left( \bigcup_{x \notin B} \gG(x) \right).
\end{equation*}
Equivalently, $T(B)$ is the set of points $(y,t) \in X^+$ such that $B(y,t) \subset B$.
From this characterisation it is clear that if $(y,t) \in T(B)$, then $t \leq r_B$, where we define
\begin{equation*}
	r_B := \sup\{r > 0 : \text{$B(y,r) \subset B$ for some $y \in X$}\}.
\end{equation*}
Note that it is possible to have $r_{B(y,t)} > t$.

Fix $q \in (0,\infty)$ and $\ga \in \RR$.
For $f \in L^0(X^+)$, define functions $\mc{A}^q f$ and $\mc{C}_\ga^q f$ on $X$ by
\begin{equation}\label{lusin}
	\mc{A}^q f(x) := \left( \iint_{\gG(x)} |f(y,t)|^q \, \frac{d\gm(y)}{V(y,t)} \, \frac{dt}{t} \right)^{1/q}
\end{equation}
and
\begin{equation}\label{carleson}
	\mc{C}_\ga^q f(x) := \sup_{B \ni x} \frac{1}{\gm(B)^\ga}\left( \frac{1}{\gm(B)} \iint_{T(B)} |f(y,t)|^q \, d\gm(y) \, \frac{dt}{t}\right)^{1/q}
\end{equation}
for all $x \in X$, where the supremum in \eqref{carleson} is taken over all balls $B \subset X$ containing $x$.
We abbreviate $\mc{C}^q := \mc{C}_0^q$.
Note that the integrals above are always defined (though possibly infinite) as the integrands are non-negative, and so we need not assume any local $q$-integrability of $f$.
We also define
\begin{equation}\label{lusin-infty}
	\mc{A}^\infty f(x) := \esssup_{(y,t) \in \gG(x)} |f(y,t)|
\end{equation}
and
\begin{equation*}
	\mc{C}_\ga^\infty f(x) := \sup_{B \ni x} \frac{1}{\gm(B)^{1 + \ga}} \esssup_{(y,t) \in T(B)} |f(y,t)|.
\end{equation*}

\begin{lem}\label{ACsemicontinuity}
	Suppose that $q \in (0,\infty]$,  $\ga \in \RR$, and $f \in L^0(X^+)$.
	Then the functions $\mc{A}^q f$ and $\mc{C}_\ga^q f$ are  lower semicontinuous.
\end{lem}

\begin{proof}
	For $q \neq \infty$ see \cite[Lemmas A.6 and A.7]{aA14}.
	It remains only to show that $\mc{A}^\infty f$ and $\mc{C}_\ga^\infty f$ are lower semicontinuous for $f \in L^0(X^+)$.
	
	For each $s > 0$ write
	\begin{equation*}
		\gG(x) + s := \{(y,t) \in X^+ : (y,t-s) \in \gG(x)\}
		= \{(y,t) \in X^+ : y \in B(x,t-s)\}.
	\end{equation*}
	Geometrically $\gG(x) + s$ is a `vertically translated' cone, and $\gG(x) + s \supset \gG(x) + r$ for all $r < s$.
	The triangle inequality implies that
	\begin{equation*}
		\gG(x) + s \subset \gG(x^\prime) \quad \text{for all $x^\prime \in B(x,s)$}.
	\end{equation*}
	
	To show that $\mc{A}^\infty f$ is lower semicontinuous, suppose that $x \in X$ and $\gl > 0$ are such that $(\mc{A}^\infty f)(x) > \gl$.
	Then the set $O := \{(y,t) \in \gG(x) : |f(y,t)| > \gl\}$ has positive measure.
	We have
	\begin{equation*}
		O = \bigcup_{n=1}^\infty O \cap (\gG(x) + n^{-1}).
	\end{equation*}
	Since the sequence of sets $O \cap (\gG(x) + n^{-1})$ is increasing in $n$, and since $O$ has positive measure, we find that there exists $n \in \NN$ such that $O \cap (\gG(x) + n^{-1})$ has positive measure.
	Thus for all $x^\prime \in B(x,n^{-1})$,
	\begin{equation*}
		\{(y,t) \in \gG(x^\prime) : |f(y,t)| > \gl\}
		\supset O \cap (\gG(x) + n^{-1})
	\end{equation*}
	has positive measure, and so $(\mc{A}^\infty f)(x^\prime) > \gl$.
	Therefore $\mc{A}^\infty f$ is lower semicontinuous.
	
	The argument for $\mc{C}_\ga^\infty$ is simpler.
	We have $(\mc{C}_\ga^\infty f)(x) > \gl$ if and only if there exists a ball $B \ni x$ such that
	\begin{equation*}
		\frac{1}{\gm(B)^{1 + \ga}} \esssup_{(y,t) \in T(B)} |f(y,t)| > \gl.
	\end{equation*}
	This immediately yields $(\mc{C}_\ga^\infty f)(x^\prime) > \gl$ for all $x^\prime \in B$, and so $\mc{C}_\ga^\infty f$ is lower semicontinuous.
\end{proof}

\begin{dfn}
	For $p \in (0,\infty)$ and $q \in (0,\infty]$, the \emph{tent space} $T^{p,q}(X)$ is the set
	\begin{equation*}
		T^{p,q}(X) := \{f \in L^0(X^+) : \mc{A}^q f \in L^p(X)\}
	\end{equation*}
	equipped with the quasi-norm
	\begin{equation*}
		\norm{f}_{T^{p,q}(X)} := \norm{\mc{A}^q f}_{L^p(X)}.
	\end{equation*}
	We define $T^{\infty,q}(X)$ by
	\begin{equation*}
		T^{\infty,q}(X) := \{f \in L^0(X^+) : \mc{C}^q f \in L^\infty(X)\}
	\end{equation*}
	equipped with the corresponding quasi-norm.
	We define $T^{\infty,\infty}(X) := L^\infty(X^+)$ with equal norms.
\end{dfn}

For the sake of notational clarity, we will write $T^{p,q}$ rather than $T^{p,q}(X)$ unless we wish to emphasise a particular choice of $X$.
Although we will always refer to tent space `quasi-norms', these are norms when $p,q \geq 1$.

\begin{rmk}
	Our definition of $\mc{A}^\infty f$ gives a function which is less than or equal to the corresponding function defined by Coifman, Meyer, and Stein \cite{CMS85}, which uses suprema instead of essential suprema.
	We also do not impose any continuity conditions in our definition of $T^{p,\infty}$.
	Therefore our space $T^{p,\infty}(\RR^n)$ is strictly larger than the Coifman--Meyer--Stein version.
	\end{rmk}

By a \emph{cylinder} we mean a subset $C \subset X^+$ of the form $C = B(x,r) \times (a,b)$ for some $(x,r) \in X^+$ and $0 < a < b < \infty$.
We say that a function $f \in L^0(X^+)$ is \emph{cylindrically supported} if it is essentially supported in a cylinder.
In general cylinders may not be precompact, and so the notion of cylindrical support is more general than that of compact support.
For all $p,q \in (0,\infty]$ we define
\begin{equation*}
	T^{p,q;c} := \{f \in T^{p,q} : \text{$f$ is cylindrically supported}\}.
\end{equation*}
and
\begin{equation*}
	L^p_c(X^+) := \{f \in L^p(X^+) : \text{$f$ is cylindrically supported}\}.
\end{equation*}

A straightforward application of the Fubini--Tonelli theorem shows that for all $q \in (0,\infty)$ and for all $f \in L^0(X^+)$,
\begin{equation*}
	\norm{f}_{T^{q,q}} = \norm{f}_{L^q(X^+)},
\end{equation*}
and so $T^{q,q} = L^q(X^+)$.
When $q = \infty$ this is true by definition.

\begin{prop}\label{density-completeness}
	For all $p,q \in (0,\infty)$, the subspace $T^{p,q;c} \subset T^{p,q}$ is dense in $T^{p,q}$.
	Furthermore, if $X$ is doubling, then for all $p,q \in (0,\infty]$, $T^{p,q}$ is complete, and when $p,q \neq \infty$, $L^q_c(X^+)$ is densely contained in $T^{p,q}$.
\end{prop}

\begin{proof}
	The second statement has already been proven in \cite[Proposition 3.5]{aA14},\footnote{The cases where $q = \infty$ are not covered there. The same proof works---the only missing ingredient is Lemma \ref{loclem}, which we defer to the end of the article.} so we need only prove the first statement.
	Suppose $f \in T^{p,q}$ and fix a point $x_0 \in X$.
	For each $k \in \NN$, define
	\begin{equation*}
		C_k := B(x_0,k) \times (k^{-1},k) \quad \text{and} \quad f_k := \mb{1}_{C_k} f.
	\end{equation*}
	Then each $f_k$ is cylindrically supported.
	We have
	\begin{align*}
		\lim_{k \to \infty} \norm{f - f_k}_{T^{p,q}}^p
		&= \lim_{k \to \infty} \int_X \mc{A}^q(\mb{1}_{C_k^c} f)(x)^p \, d\gm(x) \\
		&= \int_X \lim_{k \to \infty} \mc{A}^q(\mb{1}_{C_k^c} f)(x)^p \, d\gm(x) \\
		&= \int_X \left( \lim_{k \to \infty} \iint_{\gG(x)} |(\mb{1}_{C_k^c} f)(y,t)|^q \, \frac{d\gm(y)}{V(y, t)} \, \frac{dt}{t} \right)^{p/q} \, d\gm(x) \\
		&= \int_X \left( \iint_{\gG(x)} \lim_{k \to \infty} |(\mb{1}_{C_k^c} f)(y,t)|^q \, \frac{d\gm(y)}{V(y, t)} \, \frac{dt}{t} \right)^{p/q} \, d\gm(x) \\
		&= 0.
	\end{align*}
	All interchanges of limits and integrals follow by dominated convergence.
	Hence we have $f = \lim_{k \to \infty} f_k$, which completes the proof. 
\end{proof}

Recall the following duality from \cite[Propositions 3.10 and 3.15]{aA14}.

\begin{prop}\label{wt-duality}
	Suppose that $X$ is doubling, $p \in [1,\infty)$, and $q \in (1,\infty)$.
	Then the $L^2(X^+)$ inner product
	\begin{equation}\label{dualitypairing}
		\langle f,g \rangle := \iint_{X^+} f(x,t) \overline{g(x,t)} \, d\gm(x) \, \frac{dt}{t}
	\end{equation}
	identifies the dual of $T^{p,q}$ with $T^{p^\prime,q^\prime}$.
\end{prop}

Suppose that $p \in (0,1]$, $q \in [p,\infty]$, and $B \subset X$ is a ball.
We say that a function $a \in L^0(X^+)$ is a \emph{$T^{p,q}$ atom (associated with $B$)} if $a$ is essentially supported in $T(B)$ and if the size estimate
\begin{equation*}
	\norm{a}_{T^{q,q}} \leq \gm(B)^{\gd_{p,q}}
\end{equation*}
holds (recall that $\gd_{p,q} := q^{-1} - p^{-1}$).
A short argument shows that if $a$ is a $T^{p,q}$-atom, then $\norm{a}_{T^{p,q}} \leq 1$.

\begin{thm}[Atomic decomposition]\label{atomic-decomposition}
	Suppose that $X$ is doubling.
	Let $p \in (0,1]$ and $q \in [p,\infty]$.
	Then a function $f \in L^0(X^+)$ is in $T^{p,q}$ if and only if there exists a sequence $(a_k)_{k \in \NN}$ of $T^{p,q}$-atoms and a sequence $(\gl_k)_{k \in \NN} \in \ell^p(\NN)$ such that
	\begin{equation}\label{atomic-decompn}
		 f = \sum_{k \in \NN} \gl_k a_k
	\end{equation}
	with convergence in $T^{p,q}$.
	Furthermore, we have
	\begin{equation*}
		\norm{f}_{T^{p,q}} \simeq \inf \norm{\gl_k}_{\ell^p(\NN)},
	\end{equation*}
	where the infimum is taken over all decompositions of the form \eqref{atomic-decompn}.
\end{thm}

This is proven by Russ when $q = 2$ \cite{eR07}, and the same proof works for general $q \in [p,\infty)$.
For $q = \infty$ we need to combine the original argument of Coifman, Meyer, and Stein \cite[Proposition 2]{CMS85} with that of Russ.
We defer this to Section \ref{defprf:atomic}.

\subsection{Weighted tent spaces: definitions, duality, and atoms}\label{section:wts}

We continue to suppose that $(X,d,\gm)$ is a nondegenerate metric measure space, and again we make no further assumptions without explicit mention.

For each $s \in \RR$, we can define an operator $V^s$ on $L^0(X^+)$ by
\begin{equation*}
	(V^s f)(x,t) := V(x,t)^s f(x,t)
\end{equation*}
for all $(x,t) \in X^+$.
Note that for $r,s \in \RR$ the equality $V^r V^s = V^{r+s}$ holds, and also that $V^0$ is the identity operator.
Using these operators we define modified tent spaces, which we call \emph{weighted tent spaces}, as follows.

\begin{dfn}
	For $p \in (0,\infty)$, $q \in (0,\infty]$, and $s \in \RR$, the \emph{weighted tent space} $T^{p,q}_s$ is the set
	\begin{equation*}
		T^{p,q}_s := \{f \in L^0(X^+) : V^{-s} f \in T^{p,q}\}
	\end{equation*}
	equipped with the quasi-norm
	\begin{equation*}
		\norm{f}_{T^{p,q}_s} := \norm{V^{-s} f}_{T^{p,q}}.
	\end{equation*}
	For $q \neq \infty$, and with an additional parameter $\ga \in \RR$, we define $T^{\infty,q}_{s;\ga}$ by the quasi-norm
	\begin{equation*}
		\norm{f}_{T^{\infty,q}_{s;\ga}} := \norm{\mc{C}_\ga^q(V^{-s} f)}_{L^\infty(X)}.
	\end{equation*}
	Finally, we define $T^{\infty,\infty}_{s;\ga}$ by the norm
	\begin{equation*}
		\norm{f}_{T^{\infty,\infty}_{s;\ga}} := \sup_{B \subset X} \frac{1}{\gm(B)^{1+\ga}}\norm{V^s g}_{L^\infty(T(B))},
	\end{equation*}
	where the supremum is taken over all balls $B \subset X$.
	Note that $T^{\infty,q}_{0;0} = T^{\infty,q}$.
	We write $T^{\infty,q}_s := T^{\infty,q}_{s;0}$.
\end{dfn}

\begin{rmk}
	The weighted tent space quasi-norms of Hofmann, Mayboroda, and McIntosh \cite[\textsection 8.3]{HMM11} (with $p \neq \infty$) and Huang \cite{yH16} (including $p = \infty$ with $\ga = 0$) are given by
	\begin{equation}\label{twts}
		\norm{f}_{T^{p,q}_s(\RR^n)} := \norm{(y,t) \mapsto t^{-s} f(y,t)}_{T^{p,q}(\RR^n)},
	\end{equation}
	which are equivalent to those of our spaces $T^{p,q}_{s/n}(\RR^n)$.
	In general, when $X$ is unbounded and AD-regular of dimension $n$, the quasi-norm in \eqref{twts} (with $X$ replacing $\RR^n$) is equivalent to that of our $T^{p,q}_{s/n}$.
	We have chosen the convention of weighting with ball volumes, rather than with the variable $t$, because this leads to more geometrically intrinsic function spaces and supports embedding theorems under weaker assumptions.
\end{rmk}

For all $r,s \in \RR$, the operator $V^r$ is an isometry from $T^{p,q}_s$ to $T^{p,q}_{s + r}$.
The operator $V^{-r}$ is also an isometry, now from $T^{p,q}_{s + r}$ to $T^{p,q}_s$, and so for fixed $p$ and $q$ the weighted tent spaces $T^{p,q}_s$ are isometrically isomorphic for all $s \in \RR$.
Thus by Proposition \ref{density-completeness}, when $X$ is doubling, the spaces $T^{p,q}_s$ are all complete.

Recall the $L^2(X^+)$ inner product \eqref{dualitypairing}, which induces a duality pairing between $T^{p,q}$ and $T^{p^\prime, q^\prime}$ for appropriate $p$ and $q$ when $X$ is doubling.
For all $s \in \RR$ and all $f,g \in L^2(X^+)$ we have the equality
\begin{equation}\label{Vadjoint}
	\langle f,g \rangle = \langle V^{-s} f, V^{s} g \rangle,
\end{equation}
which yields the following duality result.

\begin{prop}\label{wtduality}
	Suppose that $X$ is doubling, $p \in [1,\infty)$, $q \in (1,\infty)$, and $s \in \RR$.
	Then the $L^2(X^+)$ inner product \eqref{dualitypairing} identifies the dual of $T^{p,q}_s$ with $T^{p^\prime, q^\prime}_{-s}$.
\end{prop}

\begin{proof}
	If $f \in T^{p,q}_s$ and $g \in T^{p^\prime, q^\prime}_{-s}$, then we have $V^{-s} f \in T^{p,q}$ and $V^s g \in T^{p^\prime, q^\prime}$, so by Proposition \ref{wt-duality} and \eqref{Vadjoint} we have
	\begin{equation*}
		|\langle f,g \rangle|  \lesssim \norm{V^{-s} f}_{T^{p,q}} \norm{V^{s} g}_{T^{p^\prime,q^\prime}} = \norm{f}_{T^{p,q}_s} \norm{g}_{T^{p^\prime, q^\prime}_{-s}}.
	\end{equation*}
	
	Conversely, if $\gf \in (T^{p,q}_s)^\prime$, then the map $\td{f} \mapsto \gf(V^s \td{f})$ determines a bounded linear functional on $T^{p,q}$ with norm dominated by $\norm{\gf}$.
	Hence by Proposition \ref{wt-duality} there exists a function $\td{g} \in T^{p^\prime,q^\prime}$ with $\norm{\td{g}}_{T^{p^\prime,q^\prime}} \lesssim \norm{\gf}$  such that
		\begin{equation*}
			\gf(f) = \gf(V^s (V^{-s} f)) = \langle V^{-s} f, \td{g}\rangle = \langle f, V^{-s} \td{g} \rangle
		\end{equation*}
	for all $f \in T^{p,q}_s$.
	Since
	\begin{equation*}
		\norm{V^{-s} \td{g}}_{T^{p^\prime,q^\prime}_{-s}} = \norm{\td{g}}_{T^{p^\prime,q^\prime}} \lesssim \norm{\gf},
	\end{equation*}
	we are done.
\end{proof}
 
There is also a duality result for $p < 1$ which incorporates the spaces $T^{\infty,q}_{s;\ga}$ with $\ga > 0$.
Before we can prove it, we need to discuss atomic decompositions.

Suppose that $p \in (0,1]$, $q \in [p,\infty]$, $s \in \RR$, and $B \subset X$ is a ball.
We say that a function $a \in L^0(X^+)$ is a \emph{$T^{p,q}_s$-atom (associated with $B$)} if $V^{-s} a$ is a $T^{p,q}$-atom.
This is equivalent to demanding that $a$ is essentially supported in $T(B)$ and that
\begin{equation*}
	\norm{a}_{T^{q,q}_s} \leq \gm(B)^{\gd_{p,q}}.
\end{equation*}
The atomic decomposition theorem for unweighted tent spaces (Theorem \ref{atomic-decomposition}) immediately implies its weighted counterpart.

\begin{prop}[Atomic decomposition for weighted tent spaces]
	Suppose that $X$ is doubling.
	Let $p \in (0,1]$, $q \in [p,\infty]$, and $s \in \RR$.
	Then a function $f \in L^0(X^+)$ is in $T^{p,q}_s$ if and only if there exists a sequence $(a_k)_{k \in \NN}$ of $T^{p,q}_s$-atoms and a sequence $(\gl_k)_{k \in \NN} \in \ell^p(\NN)$ such that
	\begin{equation}\label{weighted-atomic-decompn}
		 f = \sum_{k \in \NN} \gl_k a_k
	\end{equation}
	with convergence in $T^{p,q}_s$.
	Furthermore, we have
	\begin{equation*}
		\norm{f}_{T_s^{p,q}} \simeq \inf \norm{\gl_k}_{\ell^p(\NN)},
	\end{equation*}
	where the infimum is taken over all decompositions of the form \eqref{weighted-atomic-decompn}.
\end{prop}

Using this, we can prove the following duality result for $p < 1$.

\begin{thm}\label{duality-qb}
	Suppose that $X$ is doubling, $p \in (0,1)$, $q \in [1,\infty)$, and $s \in \RR$.
	Then the $L^2(X^+)$ inner product \eqref{dualitypairing} identifies the dual of $T^{p,q}_s$ with $T^{\infty,q^\prime}_{-s;\gd_{1,p}}$.
\end{thm}

\begin{proof}
	First suppose that $a$ is a $T^{p,q}_s$-atom associated with a ball $B \subset X$, and that $g \in T^{\infty,q^\prime}_{-s,\gd_{1,p}}$.
	Then we have
	\begin{align*}
		|\langle a,g \rangle|
		&\leq \iint_{T(B)} |V^{-s} a(y,t)| |V^s g(y,t)| \, d\gm(y) \, \frac{dt}{t} \\
		&\leq \norm{a}_{T^{q,q}_s} \gm(B)^{1/q^\prime} \gm(B)^{\gd_{1,p}} \norm{g}_{T^{\infty,q^\prime}_{-s,\gd_{1,p}}} \\
		&\leq \gm(B)^{\gd_{p,q} + \gd_{q,1} + \gd_{1,p}} \norm{g}_{T^{\infty,q^\prime}_{-s,\gd_{1,p}}} \\
		&= \norm{g}_{T^{\infty,q^\prime}_{-s,\gd_{1,p}}}.
	\end{align*}
	For general $f \in T^{p,q}_s$ we write $f$ as a sum of $T^{p,q}_s$-atoms as in \eqref{weighted-atomic-decompn} and get
	\begin{equation*}
		|\langle f,g \rangle| \leq \norm{g}_{T^{\infty,q^\prime}_{-s,\gd_{1,p}}} \norm{\gl}_{\ell^1} \leq \norm{g}_{T^{\infty,q^\prime}_{-s,\gd_{1,p}}} \norm{\gl}_{\ell^p}
	\end{equation*}
	using that $p < 1$.
	Taking the infimum over all atomic decompositions completes the argument.
	
	Conversely, suppose that $\gf \in (T^{p,q}_s)^\prime$.
	Exactly as in the classical duality proof (see \cite[Proof of Proposition 3.10]{aA14}), using the doubling assumption, there exists a function $g \in L^{q^\prime}_\text{loc}(X^+)$ such that
	\begin{equation*}
		\gf(f) = \langle f,g \rangle
	\end{equation*}
	for all $f \in T^{p,q;c}_s$.
	To show that $g$ is in $T^{\infty,q^\prime}_{-s,\gd_{1,p}}$, we estimate $\norm{V^{s} g}_{L^{q^\prime}(T(B))}$ for each ball $B \subset X$ by duality:
	\begin{align*}
		\norm{V^{s} g}_{L^{q^\prime}(T(B))}
		&= \sup_{f \in L^q(T(B))} |\langle f,V^{s} g \rangle| \norm{f}_{L^q(T(B))}^{-1} \\
		&= \sup_{f \in L_c^q(T(B))} |\langle V^{s} f, g \rangle| \norm{f}_{L^q(T(B))}^{-1}.
	\end{align*}
	H\"older's inequality implies that
	\begin{equation*}
		\norm{V^s f}_{T^{p,q}_s} \leq \gm(B)^{\gd_{q,p}} \norm{f}_{L^q(T(B))}
	\end{equation*}
	when $f$ is essentially supported in $T(B)$, so we have
	\begin{align*}
		\norm{V^s g}_{L^{q^\prime}(T(B))} \leq \gm(B)^{\gd_{q,p}} \norm{\gf}_{(T^{p,q}_s)^\prime},
	\end{align*}
	and therefore
	\begin{align*}
		\norm{g}_{T^{\infty,q^\prime}_{-s,\gd_{1,p}}} &= \sup_{B \subset X} \gm(B)^{\gd_{p,1} - (1/q^\prime)} \norm{V^s g}_{L^{q^\prime}(T(B))} \\
		&\leq \norm{\gf}_{(T^{p,q}_s)^\prime} \sup_{B \subset X} \gm(B)^{\gd_{p,1} + \gd_{1,q} + \gd_{q,p}} \\
		&= \norm{\gf}_{(T^{p,q}_s)^\prime},
	\end{align*}
	which completes the proof.
\end{proof}

\begin{rmk}
	Note that $q=1$ is included here, and excluded in the other duality results of this article.
	Generally the spaces $T^{p,q}$ with $p \leq q$ are easier to handle than those with $p > q$.
\end{rmk}

We end this section by detailing a technique, usually referred to as `convex reduction', which is very useful in relating tent spaces to each other.
Suppose $f \in L^0(X^+)$ and $M > 0$.
We define a function $f^M \in L^0(X^+)$ by
\begin{equation*}
	(f^M)(x,t) := |f(x,t)|^M
\end{equation*}
for all $(x,t) \in X^+$.
For all $q \in (0,\infty]$ and $s \in \RR$ we then have
\begin{equation*}
	\mc{A}^q(V^{-s} f^M) = \mc{A}^{Mq}(V^{-s/M} f)^M,
\end{equation*}
and for $\ga \in \RR$ we also have
\begin{equation*}
	\mc{C}^q_\ga (V^{-s} f^M) = \mc{C}^{Mq}_{\ga/M} (V^{-s/M} f)^M.
\end{equation*}
Therefore, for $p \in (0,\infty)$ we have
\begin{align*}
	\norm{f^M}_{T^{p,q}_s} &= \norm{\mc{A}^{Mq}(V^{-s/M} f)^M}_{L^p(X)} \\
	&=\norm{\mc{A}^{Mq}(V^{-s/M} f)}_{L^{Mp}(X)}^{M}\\
	&= \norm{f \mid T^{Mp,Mq}_{s/M}}^{M},
\end{align*}
and likewise for $p = \infty$ and $q < \infty$ we have
\begin{equation*}
	\norm{f^M}_{T^{\infty,q}_{s,\ga}} = \norm{f \mid T^{\infty,Mq}_{s/M,\ga/M}}^M 
\end{equation*}
The case $p = q = \infty$ behaves in the same way:
\begin{equation*}
	\norm{f^M}_{T^{\infty,\infty}_{s}} = \norm{(V^{-s/M} f)^M}_{L^\infty(X^+)}
	= \norm{f}_{T^{\infty,\infty}_{s/M}}^M.
\end{equation*}
These equalities often allow us to deduce properties of $T^{p,q}_s$ from properties of $T^{Mp,Mq}_{s/M}$, and vice versa.
We will use them frequently.

\section{Interpolation and embeddings}

As always, we assume that $(X,d,\gm)$ is a nondegenerate metric measure space.
We will freely use notation and terminology regarding interpolation theory; the uninitiated reader may refer to Bergh and L\"ofstr\"om \cite{BL76}.

\subsection{Complex interpolation}\label{cint}

In this section we will make the following identification of the complex interpolants of weighted tent spaces in the Banach range of exponents.

\begin{thm}\label{cpxinterp}
	Suppose that $X$ is doubling, $p_0,p_1 \in [1,\infty]$ (not both $\infty$), $q_0,q_1 \in (1,\infty)$, $s_0, s_1 \in \RR$, and $\gq \in (0,1)$.
	Then we have the identification
	\begin{equation*}
		[T^{p_0,q_0}_{s_0}, T^{p_1,q_1}_{s_1}]_\gq = T^{p_\gq, q_\gq}_{s_\gq}
	\end{equation*}
	with equivalent norms, where $p_\gq^{-1} = (1-\gq)p_0^{-1} + \gq p_1^{-1}$, $q_\gq^{-1} = (1-\gq)q_0^{-1} + \gq q_1^{-1}$, and $s_\gq = (1-\gq)s_0 + \gq s_1$.
\end{thm}

\begin{rmk}\label{cpxint-refs}
In the case where $X = \RR^n$ with the Euclidean distance and Lebesgue measure, this result (with $p_0,p_1 < 1$ permitted) is due to Hofmann, Mayboroda, and McIntosh \cite[Lemma 8.23]{HMM11}.
A more general result, still with $X = \RR^n$, is proven by Huang \cite[Theorem 4.3]{yH16} with $q_0,q_1 = \infty$ also permitted, and with Whitney averages incorporated.
Both of these results are proven by means of factorisation theorems for weighted tent spaces (with Whitney averages in the second case), and by invoking an extension of Calder\'on's product formula to quasi-Banach spaces due to Kalton and Mitrea \cite[Theorem 3.4]{KM98}.
We have chosen to stay in the Banach range with $1 < q_0,q_1 < \infty$ for now, as establishing a general factorisation result would take us too far afield.
\end{rmk}

Note that if $p_0 = \infty$ (say) then we are implicitly considering $T^{\infty,q_0}_{s_0;\ga}$ with $\ga = 0$; interpolation of spaces with $\ga \neq 0$ is not covered by this theorem.
This is because the method of proof uses duality, and to realise $T^{\infty,q_0}_{s_0;\ga}$ with $\ga \neq 0$ as a dual space we would need to deal with complex interpolation of quasi-Banach spaces, which adds difficulties that we have chosen to avoid.

Before moving on to the proof of Theorem \ref{cpxinterp}, we must fix some notation.
For $q \in (1,\infty)$ and $s \in \RR$, write
\begin{equation}\label{wtlq}
	L^q_s(X^+) := L^q(X^+, V^{-qs-1}) := L^q\left(X^+, V^{-qs}(y,t) \, \frac{d\gm(y)}{V(y,t)} \, \frac{dt}{t} \right)
\end{equation}
(this notation is consistent with viewing the function $V^{-qs-1}$ as a weight on the product measure $d\gm \, dt/t$).

An important observation, originating from Harboure, Torrea, and Viviani \cite{HTV91}, is that for all $p \in [1,\infty)$, $q \in (1,\infty)$ and $s \in \RR$, one can write
\begin{equation*}
	\norm{f}_{T^{p,q}_s} = \norm{Hf \mid L^p(X:L_s^q(X^+))} 
\end{equation*}
for $f \in L^0(X^+)$, where
\begin{equation*}
	Hf(x) = \mb{1}_{\gG(x)} f.
\end{equation*}
Hence $H$ is an isometry from $T^{p,q}_s$ to $L^p(X:L^q_s(X^+))$.
Because of the restriction on $q$, the theory of Lebesgue spaces (more precisely, Bochner spaces) with values in reflexive Banach spaces is then available to us.

This proof follows previous arguments of the author \cite{aA14}, which are based on the ideas of Harboure, Torrea, and Viviani \cite{HTV91} and of Bernal \cite{aB92}, with only small modifications to incorporate additional parameters.
We include it to show where these modifications occur: in the use of duality, and in the convex reduction.

\begin{proof}[Proof of Theorem \ref{cpxinterp}]
	First we will prove the result for $p_0,p_1 \in (1,\infty)$.
	Since $H$ is an isometry from $T^{p_j,q_j}_{s_j}$ to $L^{p_j}(X : L_{s_j}^{q_j}(X^+))$ for $j = 0,1$, the interpolation property implies that $H$ is bounded (with norm $\leq 1$ due to exactness of the complex interpolation functor)
	\begin{equation*}
		[T^{p_0,q_0}_{s_0}, T^{p_1,q_1}_{s_1}]_\gq \to L^{p_\gq}\left(X : [L^{q_0}_{s_0}(X^+), L_{s_1}^{q_1}(X^+)]_\gq\right).
	\end{equation*}
	Here we have used the standard identification of complex interpolants of Banach-valued Lebesgue spaces \cite[Theorem 5.1.2]{BL76}.
	The standard identification of complex interpolants of weighted Lebesgue spaces \cite[Theorem 5.5.3]{BL76} gives
	\begin{equation*}
		[L_{s_0}^{q_0}(X^+), L_{s_1}^{q_1}(X^+)]_\gq = L_{s_\gq}^{q_\gq}(X^+), 
	\end{equation*}
	and we conclude that
	\begin{align*}
		\norm{f}_{T^{p_\gq,q_\gq}_{s_\gq}}
		&= \norm{Hf \mid L^{p_\gq}(X : L_{s_\gq}^{q_\gq}(X^+))} \\
		&\leq \norm{f \mid [T^{p_0,q_0}_{s_0}, T^{p_1,q_1}_{s_1}]_\gq}
	\end{align*}
	for all $f \in [T^{p_0,q_0}_{s_0},T^{p_1,q_1}_{s_1}]_\gq$.
	Therefore
	\begin{equation}\label{cint1}
		[T^{p_0,q_0}_{s_0}, T^{p_1,q_1}_{s_1}]_\gq \subset T^{p_\gq, q_\gq}_{s_\gq}.
	\end{equation}
	
	To obtain the reverse inclusion, we use the duality theorem for complex interpolation \cite[Theorem 4.5.1 and Corollary 4.5.2]{BL76}.
	Since $X$ is doubling, and by our restrictions on the exponents, the spaces $T^{p_0,q_0}_{s_0}$ and $T^{p_1,q_1}_{s_1}$ are reflexive (by Proposition \ref{wtduality}) with intersection dense in both spaces (as it contains the dense subspace $L^{\max(q_0,q_1)}_c(X^+)$ by Proposition \ref{density-completeness}).
	Therefore the assumptions of the duality theorem for complex interpolation are satisfied, and we have
	\begin{align*}
		T^{p_\gq, q_\gq}_{s_\gq} &= (T^{p_\gq^\prime, q_\gq^\prime}_{-s_\gq})^\prime \\
		&\subset [T^{p_0^\prime, q_0^\prime}_{-s_0}, T^{p_1^\prime, q_1^\prime}_{-s_1}]_\gq^\prime \\
		&= [T^{p_0,q_0}_{s_0}, T^{p_1,q_1}_{s_1}]_\gq
	\end{align*}
	where the first two lines follow from Proposition \eqref{wtduality} and \eqref{cint1}, and the third line uses the duality theorem for complex interpolation combined with Proposition \ref{wtduality}.
	
	We can extend this result to $p_0,p_1 \in [1,\infty]$ using the technique of \cite[Proposition 3.18]{aA14}. The argument is essentially identical, so we will not include the details here.
\end{proof}
	
\subsection{Real interpolation: the reflexive range}\label{sec:real-int}

In order to discuss real interpolation of weighted tent spaces, we need to introduce a new scale of function spaces, which we denote by $Z^{p,q}_s = Z^{p,q}_s(X)$.\footnote{We use this notation because almost every other reasonable letter seems to be taken.}

\begin{dfn}\label{Zdfn}
For $c_0 \in (0,\infty)$, $c_1 \in (1,\infty)$, and $(x,t) \in X^+$, we define the \emph{Whitney region}
\begin{equation*}
	\gO_{c_0,c_1}(x,t) := B(x,c_0 t) \times (c_1^{-1} t, c_1 t) \subset X^+,
\end{equation*}
and for $q \in (0,\infty)$, $f \in L^0(X^+)$, and $(x,t) \in X^+$ we define the \emph{$L^q$-Whitney average}
\begin{equation*}
	(\mc{W}^q_{c_0,c_1} f)(x,t) := \left( \bariint_{\gO_{c_0,c_1}(x,t)}^{} |f(\gx,\gt)|^q \, d\gm(\gx) \, d\gt \right)^{1/q}.
\end{equation*}

For $p,q \in (0,\infty)$, $s \in \RR$, $c_0 \in (0,\infty)$, $c_1 \in (1,\infty)$, and $f \in L^0(X^+)$, we then define the quasi-norm
\begin{equation*}
	\norm{f}_{Z^{p,q}_s(X;c_0,c_1)} := \norm{\mc{W}_{c_0,c_1}^q (V^{-s} f)}_{L^p(X^+)}.
\end{equation*}
and the \emph{$Z$-space}
\begin{equation*}
	Z^{p,q}_s(X;c_0,c_1) := \{f \in L^0(X^+) : \norm{f}_{Z^{p,q}_s(X;c_0,c_1)} < \infty\}.
\end{equation*}
\end{dfn}

In this section we will prove the following theorem, which identifies real interpolants of weighted tent spaces \emph{in the reflexive range}.
We will extend this to the full range of exponents in the Euclidean case in the next section.

\begin{thm}\label{rint}
	Suppose that $X$ is AD-regular and unbounded, $p_0, p_1,q \in (1,\infty)$, $s_0 \neq s_1 \in \RR$, and $\gq \in (0,1)$.
	Then for any $c_0 \in (0,\infty)$ and $c_1 \in (1,\infty)$ we have the identification
	\begin{equation}\label{eqn:rintidentif}
		(T^{p_0,q}_{s_0},T^{p_1,q}_{s_1})_{\gq,p_\gq} = Z^{p_\gq,q}_{s_\gq}(X;c_0,c_1)
	\end{equation}
	with equivalent norms, where $p_\gq^{-1} = (1-\gq)p_0^{-1} + \gq p_1^{-1}$ and $s_\gq = (1-\gq)s_0 + \gq s_1$.
\end{thm}

As a corollary, in the case when $X$ is AD-regular and unbounded, and when $p, q > 1$, the spaces $Z^{p,q}_s(X;c_0,c_1)$ are independent of the parameters $(c_0,c_1)$ with equivalent norms, and we can denote them all simply by $Z^{p,q}_s$.\footnote{One can  prove independence of the parameters $(c_0,c_1)$ directly when $X$ is doubling, but proving this here would take us even further off course.}
We remark that most of the proof does not require AD-regularity, but in its absence we obtain identifications of the real interpolants which are less convenient.

The proof relies on the following identification of real interpolants of weighted $L^q$ spaces, with fixed $q$ and distinct weights, due to Gilbert \cite[Theorem 3.7]{jG72}.
The cases $p \leq 1$ and $q < 1$ are not considered there, but the proof still works without any modifications in these cases.
Note that the original statement of this theorem contains a sign error in the expression corresponding to \eqref{gilbertdisc}.

\begin{thm}[Gilbert]\label{gilbert}
	Suppose $(M,\gm)$ is a $\gs$-finite measure space and let $w$ be a weight on $(M,\gm)$.
	Let $p,q \in (0,\infty)$ and $\gq \in (0,1)$.
	For all $r \in (1,\infty)$, and for $f \in L^0(M)$, the expressions
	\begin{equation}\label{gilbertdisc}
		\norm{\left(r^{-k\gq} \norm{\mb{1}_{x : w(x) \in (r^{-k},r^{-k+1}]}f}_{L^q(M)}\right)_{k \in \ZZ}}_{\ell^p(\ZZ)}
	\end{equation}
	\begin{equation}\label{gilbert2}
		\norm{s^{1-\gq} \norm{\mb{1}_{x : w(x) \leq 1/s} f}_{L^q(M,w^q)}}_{L^p(\RR_+, ds/s)}
	\end{equation}
	and
	\begin{equation}\label{gilbert3}
		\norm{s^{-\gq} \norm{\mb{1}_{x : w(x) > 1/s} f}_{L^q(M)}}_{L^p(\RR_+, ds/s)}
	\end{equation}
	define equivalent norms on the real interpolation space
	\begin{equation*}
		(L^q(M), L^q(M,w^q))_{\gq,p}.
	\end{equation*}
\end{thm}

The first step in the proof of Theorem \ref{rint} is a preliminary identification of the real interpolation norm.

\begin{prop}\label{rintprop}
	Let all numerical parameters be as in the statement of Theorem \ref{rint}.
	Then for all $f \in L^0(X^+)$ we have the equivalence
	\begin{equation}\label{eqn:rintprop}
		\norm{f \mid (T_{s_0}^{p_0,q}, T_{s_1}^{p_1,q})_{\gq,p_\gq}} \simeq \norm{x \mapsto \norm{\mb{1}_{\gG(x)} f \mid (L_{s_0}^q(X^+), L_{s_1}^q(X^+))_{\gq,p_\gq}}}_{L^{p_\gq}(X)}.
	\end{equation}
\end{prop}

\begin{proof}
	We use the notation of the previous section.
	We have already noted that the map $\map{H}{T_s^{p,q}}{L^p(X : L_s^q(X^+))}$ with $Hf(x) = \mb{1}_{\gG(x)}f$ is an isometry.
	Furthermore, as shown in \cite{aA14} (see the discussion preceding Proposition 3.12 there), $H(T_s^{p,q})$ is complemented in $L^p(X : L_s^q(X^+))$, and there is a common projection onto these spaces.
	Therefore we have (by \cite[Theorem 1.17.1.1]{hT78} for example)
	\begin{equation*}
		\norm{ f \mid (T_{s_0}^{p_0,q}, T_{s_1}^{p_1,q})_{\gq,p_\gq} } \simeq \norm{Hf \mid (L^{p_0}(X : L_{s_0}^{q}(X^+)), L^{p_1}(X : L_{s_1}^{q}(X^+)))_{\gq,p_\gq}}.
	\end{equation*}
	The Lions--Peetre result on real interpolation of Banach-valued Lebesgue spaces (see for example \cite[Remark 7]{gP93}) then implies that
	\begin{equation*}
		\norm{ f \mid (T_{s_0}^{p_0,q}, T_{s_1}^{p_1,q})_{\gq,p_\gq} } \simeq \norm{Hf \mid (L^{p_\gq} (X : (L_{s_0}^q(X^+), L_{s_1}^q(X^+))_{\gq,p_\gq})}.
	\end{equation*}
	Since $Hf(x) = \mb{1}_{\gG(x)}f$, this proves \eqref{eqn:rintprop}.
\end{proof}

Having proven Proposition \ref{rintprop}, we can use Theorem \ref{gilbert} to provide some useful characterisations of the real interpolation norm.
For $f \in L^0(X^+)$ and $a,b \in [0,\infty]$, we define the truncation
\begin{align*}
	f_{a,b} := \mb{1}_{X \times (a,b)} f.
\end{align*}
Note that in this theorem we allow for $p_0,p_1 \leq 1$; we will use this range of exponents in the next section.

\begin{thm}\label{gilbcor}
	Suppose $p_0, p_1,q \in (0,\infty)$, $s_0 \neq s_1 \in \RR$, and $\gq \in (0,1)$, and suppose that $X$ is AD-regular of dimension $n$ and unbounded.
	Let $r \in (1,\infty)$.
	Then for $f \in L^0(X^+)$ we have norm equivalences
	\begin{align}
		&\norm{x \mapsto \norm{\mb{1}_{\gG(x)} f \mid (L_{s_0}^q(X^+), L_{s_1}^q(X^+))_{\gq,p_\gq}}}_{L^{p_\gq}(X)} \nonumber \\
		&\simeq \norm{\gt^{n(s_1-s_0)(1-\gq)} \norm{f_{\gt,\infty}}_{T_{s_1}^{p_\gq,q}}}_{L^{p_\gq}(\RR_+, d\gt/\gt)} \label{inftynorm}\\
		&\simeq \norm{\gt^{-n(s_1-s_0)\gq} \norm{f_{0,\gt}}_{T_{s_0}^{p_\gq,q}} }_{L^{p_\gq}(\RR_+, d\gt/\gt)} \label{0norm}\\
		&\simeq_r \norm{ (r^{-nk\gq(s_1 - s_0)} \norm{f_{r^{-k},r^{-k+1}}}_{T^{p_\gq,q}_{s_0}})_{k \in \ZZ}}_{\ell^{p_\gq}(\ZZ)}. \label{seqnorm}
	\end{align}
\end{thm}

\begin{proof}
	First assume that $s_1 > s_0$.
	Let $\gm_{s_0}^q$ be the measure on $X^+$ given by
	\begin{equation*}
		d\gm_{s_0}^q(y,t) := t^{-qs_0 n} \frac{d\gm(y) \, dt}{V(y,t) t}.
	\end{equation*}
	Since $X$ is AD-regular of dimension $n$ and unbounded, we have that $\norm{f}_{L^q(\gm_{s_0}^q)} \simeq \norm{f}_{L_{s_0}^q(X^+)}$.
	Also define the weight $w(y,t) := t^{-(s_1 - s_0)n}$, so that $w^q \gm_{s_0}^q = \gm_{s_1}^q$.
	
	We will obtain the norm equivalence \eqref{seqnorm}.
	For $1 < r < \infty$ and $k \in \ZZ$, we have $r^{-k} < w(y,t) \leq r^{-k+1}$ if and only if $t \in [r^{(k-1)/n(s_1-s_0)}, r^{k/n(s_1-s_0)})$ (here we use $s_1 > s_0$).
	Using the characterisation \eqref{gilbertdisc} of Theorem \ref{gilbert}, and replacing $r$ with $r^{n (s_1-s_0)}$, for $f \in L^0(X^+)$ we have
	\begin{align*}
		&\norm{x \mapsto \norm{\mb{1}_{\gG(x)} f \mid (L_{s_0}^q(X^+), L_{s_1}^q(X^+))_{\gq,p_\gq}}}_{L^{p_\gq}(X)} \\
		&\simeq \left( \int_X \norm{\mb{1}_{\gG(x)} f}_{(L^q(\gm_{s_0}^q), L^q(w^q \gm_{s_0}^q))_{\gq,p_\gq}}^{p_\gq} \, d\gm(x) \right)^{1/p_\gq} \\
		&\simeq \left( \int_X \sum_{k \in \ZZ} r^{-n(s_1 - s_0)k \gq p_\gq} \norm{\mb{1}_{\gG(x)} f_{r^{k-1}, r^k}}_{L^q(\gm_{s_0}^q)}^{p_\gq} \, d\gm(x) \right)^{1/p_\gq} \\
		&\simeq \left( \sum_{k \in \ZZ} r^{-n(s_1-s_0)k\gq p_\gq} \int_X \mc{A}^{q}(V^{-s_0} f_{r^{k-1},r^k})(x)^{p_\gq} \, d\gm(x) \right)^{1/p_\gq} \\
		&= \norm{(r^{-n(s_1-s_0)k\gq} \norm{f_{r^{k-1},r^k}}_{T^{p_\gq,q}_{s_0}})_{k \in \ZZ}}_{\ell^{p_\gq}(\ZZ)}.
	\end{align*}
	This proves the norm equivalence \eqref{seqnorm} for all $f \in L^0(X^+)$ when $s_1 > s_0$.
	If $s_1 < s_0$, one simply uses the identification $(L_{s_0}^q(X^+), L_{s_1}^q(X^+))_{\gq,p_\gq} = (L_{s_1}^q(X^+), L_{s_0}^q(X^+))_{1-\gq,p_\gq}$ \cite[Theorem 3.4.1(a)]{BL76} to reduce the problem to the case where $s_0 < s_1$. 
	
	The equivalences \eqref{inftynorm} and \eqref{0norm} follow from the characterisations \eqref{gilbert2} and \eqref{gilbert3} of Theorem \ref{gilbert} in the same way, with integrals replacing sums throughout.
	We omit the details here.
\end{proof}

Finally we can prove the main theorem: the identification of the real interpolants of weighted tent spaces as $Z$-spaces.

\begin{proof}[Proof of Theorem \ref{rint}.]
	Suppose $f \in L^0(X^+)$.
	Using the characterisation \eqref{seqnorm} in Theorem \ref{gilbcor} with $r = c_1 > 1$, and using aperture $c_0/c_1$ for the tent space (making use of the change of aperture theorem \cite[Proposition 3.21]{aA14}), we have
	\begin{align*}
		&\norm{f \mid (T^{p_0,q}_{s_0}, T^{p_1,q}_{s_1})_{\gq,p_\gq}}^{p_\gq} \\
		&\simeq \sum_{k \in \ZZ} c_1^{-n(s_1-s_0)k\gq p_\gq} \int_X \left( \int_{c_1^{k-1}}^{c_1^k} \int_{B(x,c_0 t / c_1)} |t^{-n s_0} f(y,t)|^q \, \frac{d\gm(y)}{V(y,t)} \, \frac{dt}{t} \right)^{p_\gq / q} \, d\gm(x) \\
		&\simeq \int_X \sum_{k \in \ZZ} c_1^{-n(s_1-s_0)k\gq p_\gq} \cdot  \int_{c_1^{k-1}}^{c_1^{k}} \left( \int_{c_1^{k-1}}^{c_1^{k}} \int_{B(x,c_0 t/c_1)} |t^{-n s_0} f(y,t)|^q \, \frac{d\gm(y)}{V(y,t)} \, \frac{dt}{t} \right)^{p_\gq /q} \, \frac{dr}{r} \, d\gm(x) \\
		&\lesssim \int_X \sum_{k \in \ZZ} c_1^{-n(s_1-s_0)k\gq p_\gq} \int_{c_1^{k-1}}^{c_1^{k}} \left( \bariint_{\gO_{c_0,c_1}(x,r)}^{} |r^{-n s_0} f(y,t)|^q \, d\gm(y) \, dt \right)^{p_\gq / q} \, \frac{dr}{r} \, d\gm(x) \\
		&\simeq \int_X \int_0^\infty r^{-n(s_1-s_0)\gq p_\gq} \left( \bariint_{\gO_{c_0,c_1}(x,r)}^{} |r^{-n s_0} f|^q \right)^{p_\gq / q} \, \frac{dr}{r} \, d\gm(x) \\
		&= \iint_{X^+} \left( \bariint_{\gO_{c_0,c_1}(x,r)}^{} |r^{-n s_\gq} f|^q \right)^{p_\gq / q} \, d\gm(x) \, \frac{dr}{r} \\
		&\simeq \norm{f}_{Z^{p_\gq, q}_{s_\gq}(X;c_0,c_1)}^{p_\gq},
	\end{align*}
	using that $B(x,c_0 t/c_1) \times (c_1^{k-1}, c_1^k) \subset \gO_{c_0,c_1}(x,r)$ whenever $r \in (c_1^{k-1}, c_1^k)$.
		
	To prove the reverse estimate we use the same argument, this time using that for $r,t \in (2^{k-1},2^k)$ we have $\gO_{c_0,c_1}(x,t) \subset B(x, 2c_0 t) \times (c_1^{-1} 2^{k-1}, c_1 2^k)$.
	Using aperture $2c_0$ for the tent space, we can then conclude that
	\begin{align*}
		&\norm{f}_{Z^{p_\gq,q}_{s_\gq}(X;c_0,c_1)}^{p_\gq} \\
		&\simeq \int_X \sum_{k \in \ZZ} 2^{-n(s_1 - s_0)k\gq p_\gq} \int_{2^{k-1}}^{2^{k}} \left( \bariint_{\gO_{c_0,c_1}(x,r)}^{} |r^{-n s_0} f|^q \right)^{p_\gq / q} \, \frac{dr}{r} \, d\gm(x) \\
		&\lesssim \int_X \sum_{k \in \ZZ} 2^{-n(s_1 - s_0)k\gq p_\gq}  \int_{2^{k-1}}^{2^{k}} \left( \int_{c_1^{-1} 2^{k-1}}^{c_1 2^{k}} \int_{B(x,2c_0t)} |r^{-ns_0} f(y,t)|^q \, \frac{d\gm(y)}{V(y,t)} \, \frac{dt}{t} \right)^{p_\gq / q} \, \frac{dr}{r} \, d\gm(x) \\
		&\simeq \int_X \sum_{k \in \ZZ} 2^{-n(s_1 - s_0)k\gq p_\gq}  \int_{c_1^{-1} 2^{k-1}}^{c_1 2^{k}} \left( \int_{c_1^{-1} 2^{k-1}}^{c_1 2^{k}} \int_{B(x,2c_0t)} |r^{-ns_0} f(y,t)|^q \, \frac{d\gm(y)}{V(y,t)} \, \frac{dt}{t} \right)^{p_\gq / q} \, \frac{dr}{r} \, d\gm(x) \\
		&\simeq \norm{f \mid (T_{s_0}^{p_0,q}, T_{s_1}^{p_1,q})_{\gq,p_\gq}}^{p_\gq}.
	\end{align*}
	This completes the proof of Theorem \ref{rint}.
\end{proof}

\begin{rmk}\label{identn-rmk}
	Note that this argument shows that 
	\begin{equation*}
		\norm{ (r^{-nk\gq(s_1 - s_0)} \norm{f_{r^{-k},r^{-k+1}}}_{T^{p_\gq,q}_{s_0}})_{k \in \ZZ}}_{\ell^{p_\gq}(\ZZ)} \simeq \norm{ f }_{Z^{p_\gq,q}_{s_\gq}(X;c_0,c_1)}
	\end{equation*}
	whenever $X$ is AD-regular of dimension $n$ and unbounded, for all $p_0,p_1 \in (0,\infty)$, $c_0 \in (0,\infty)$, and $c_1 \in (1,\infty)$.
	Therefore, since Theorem \ref{gilbcor} also holds for this range of exponents, to establish the identification \eqref{eqn:rintidentif} for $p_0,p_1 \in (0,\infty)$ it suffices to extend Proposition \ref{rintprop} to $p_0,p_1 \in (0,\infty)$.
	We will do this in the next section in the Euclidean case.
\end{rmk}

\subsection{Real interpolation: the non-reflexive range}

In this section we prove the following extension of Theorem \ref{rint}.
In what follows, we always consider $\RR^n$ as a metric measure space with the Euclidean distance and Lebesgue measure.
Throughout this section we use the real interpolation method for quasi-normed Abelian groups, as described in \cite[\textsection 3.11]{BL76}.

\begin{thm}\label{rint2}
	Suppose that $p_0, p_1 \in (0,\infty)$, $q \in [1,\infty)$, $s_0 \neq s_1 \in \RR$, and $\gq \in (0,1)$.
	Then for any $c_0 \in (0,\infty)$ and $c_1 \in (1,\infty)$ we have the identification
	\begin{equation}\label{eqn:rintidentif}
		(T^{p_0,q}_{s_0}(\RR^n),T^{p_1,q}_{s_1}(\RR^n))_{\gq,p_\gq} = Z^{p_\gq,q}_{s_\gq}(\RR^n;c_0,c_1)
	\end{equation}
	with equivalent quasi-norms, where $p_\gq^{-1} = (1-\gq)p_0^{-1} + \gq p_1^{-1}$ and $s_\gq = (1-\gq)s_0 + \gq s_1$.
\end{thm}

The main difficulty here is that vector-valued Bochner space techniques are not available to us, as we would need to use quasi-Banach valued $L^p$ spaces with $p < 1$, and such a theory is not well-developed.
Furthermore, although the weighted tent spaces $T^{p,q}_s$ embed isometrically into $L^p(X : L_s^q(X^+))$ in this range of exponents, their image may not be complemented, and so we cannot easily identify interpolants of their images.\footnote{Harboure, Torrea, and Viviani \cite{HTV91} avoid this problem by embedding $T^1$ into a vector-valued Hardy space $H^1$. If we were to extend this argument we would need identifications of quasi-Banach real interpolants of certain vector-valued Hardy spaces $H^p$ for $p \leq 1$, which is very uncertain terrain (see Blasco and Xu \cite{BX91}).}
We must argue directly.

First we recall the so-called `power theorem' \cite[Theorem 3.11.6]{BL76}, which allows us to exploit the convexity relations between weighted tent spaces.
If $A$ is a quasi-Banach space with quasi-norm $\norm{\cdot}$ and if $\gr > 0$, then $\norm{\cdot}^\gr$ is a quasi-norm on $A$ in the sense of \cite[page 59]{BL76}, and we denote the resulting quasi-normed Abelian group by $A^\gr$.

\begin{thm}[Power theorem]\label{powerthm}
	Let $(A_0,A_1)$ be a compatible couple of quasi-Banach spaces.
	Let $\gr_0, \gr_1 \in (0,\infty)$, $\gh \in (0,1)$, and $r \in (0,\infty]$, and define $\gr := (1-\gh)\gr_0 + \gh \gr_1$, $\gq := \gh \gr_1 / \gr$, and $\gs := r \gr$.
	Then we have
	\begin{equation*}
		((A_0)^{\gr_0}, (A_1)^{\gr_1})_{\gh,r} = ((A_0,A_1)_{\gq,\gs})^\gr
	\end{equation*}
	with equivalent quasi-norms.
\end{thm}

Before proving Theorem \ref{rint2} we must establish some technical lemmas.
Recall that we previously defined the spaces $L_s^q(X^+)$ in \eqref{wtlq}.

\begin{lem}\label{Klem}
	Suppose $x \in X$, $\ga \in (0,\infty)$, and let all other numerical parameters be as in the statement of Theorem \ref{rint2}.
	Then for all cylindrically supported $f \in L^0(X^+)$ we have
	\begin{align}\label{Klem1}
		K(\ga,\mb{1}_{\gG(x)}f &; L_{s_0}^q(X^+), L_{s_1}^q(X^+)) \nonumber \\
		&= \inf_{f = \gf_0 + \gf_1} \left( \mc{A}^q(V^{-s_0} \gf_0)(x) + \ga\mc{A}^q(V^{-s_1} \gf_1)(x) \right)
	\end{align}
	and
	\begin{align}\label{Klem2}
		K(\ga,\mb{1}_{\gG(x)}f &; L_{s_0}^q(X^+)^{p_0}, L_{s_1}^q(X^+)^{p_1}) \nonumber \\
		&= \inf_{f = \gf_0 + \gf_1} \left(\mc{A}^q(V^{-s_0} \gf_0)(x)^{p_0} + \ga \mc{A}^q(V^{-s_1} \gf_1)(x)^{p_1} \right)
	\end{align}
	where the infima are taken over all decompositions $f = \gf_0 + \gf_1$ in $L^0(X^+)$ with $\gf_0, \gf_1$ cylindrically supported.
\end{lem}

\begin{proof}
	We will only prove the equality \eqref{Klem1}, as the proof of \eqref{Klem2} is essentially the same.
		
		Given a decomposition $f = \gf_0 + \gf_1$ in $L^0(X^+)$, we have a corresponding decomposition $\mb{1}_{\gG(x)} f = \mb{1}_{\gG(x)} \gf_0 + \mb{1}_{\gG(x)} \gf_1$, with $\norm{\mb{1}_{\gG(x)} \gf_0}_{L_{s_0}^q(X^+)} = \mc{A}^q(V^{-s_0}\gf_0)(x)$ and likewise for $\gf_1$.
		This shows that
		\begin{equation*}
			K(\ga,\mb{1}_{\gG(x)} f ; L_{s_0}^q(X^+), L_{s_1}^q (X^+)) \leq \inf_{f = \gf_0 + \gf_1} \left( \mc{A}^{q}(V^{-s_0}\gf_0)(x) + \ga\mc{A}^{q}(V^{-s_1}\gf_1)(x) \right).
		\end{equation*}
		
		For the reverse inequality, suppose that $\mb{1}_{\gG(x)} f = \gf_0 + \gf_1$ in $L^0(X^+)$, and suppose $f$ is essentially supported in a cylinder $C$.
		Multiplication by the characteristic function $\mb{1}_{\gG(x) \cap C}$ does not increase the quasi-norms of $\gf_0$ and $\gf_1$ in $L_{s_0}^q(X^+)$ and $L_{s_1}^q(X^+)$ respectively, so without loss of generality we can assume that $\gf_0$ and $\gf_1$ are cylindrically supported in $\gG(x)$.
		Now let $f = \gy_0 + \gy_1$ be an arbitrary decomposition in $L^0(X^+)$, and define
		\begin{align*}
			\wtd{\gy_0} &:= \mb{1}_{\gG(x)} \gf_0 + \mb{1}_{X^+ \sm \gG(x)} \gy_0, \\
			\wtd{\gy_1} &:= \mb{1}_{\gG(x)} \gf_1 + \mb{1}_{X^+ \sm \gG(x)} \gy_1.
		\end{align*}
		Then $f = \wtd{\gy_0} + \wtd{\gy_1}$ in $L^0(X^+)$, and we have
		\begin{equation*}
			\mc{A}^{q}(V^{-s_0} \wtd{\gy_0})(x) = \mc{A}^{q}(V^{-s_0} \gf_0)(x) = \norm{\mb{1}_{\gG(x)} \gf_0}_{L_{s_0}^q(X^+)}
		\end{equation*}
		and likewise for $\wtd{\gy_1}$.
		The conclusion follows from the definition of the $K$-functional.
\end{proof}

\begin{lem}\label{A-continuity}
	Suppose $f \in L_c^q(X^+)$.
	Then $\mc{A}^q f$ is continuous.
\end{lem}

\begin{proof}
	Let $f$ be essentially supported in the cylinder $C := B(c,r) \times (\gk_0,\gk_1)$.
	First, for all $x \in X$ we estimate
	\begin{align*}
		\mc{A}^q f(x) &\leq \left( \iint_C |f(y,t)|^q \, \frac{d\gm(y}{V(y,t)} \, \frac{dt}{t} \right)^{1/q} \\
		&\leq \left( \inf_{y \in B} V(y,\gk_0) \right)^{-1/q} \norm{f}_{L^q(X^+)} \\
		&\lesssim \norm{f}_{L^q(X^+)},
	\end{align*}
	using the estimate \eqref{inftyclaim} from the proof of Lemma \ref{loclem}.
	
	For all $x \in X$ we thus have
	\begin{align*}
		\lim_{z \to x} \left| \mc{A}^q f(x) - \mc{A}^q f(z) \right|
		&\leq \lim_{z \to x} \left( \iint_{X^+} |\mb{1}_{\gG(x)} - \mb{1}_{\gG(z)}| |f(y,t)|^q \, \frac{d\gm(y)}{V(y,t)} \, \frac{dt}{t}\right)^{1/q} = 0
	\end{align*}
	by dominated convergence, since $\mb{1}_{\gG(x)} - \mb{1}_{\gG(z)} \to 0$ pointwise as $z \to x$, and since
	\begin{align*}
		\left( \iint_{X^+} |\mb{1}_{\gG(x)} - \mb{1}_{\gG(z)}| |f(y,t)|^q \, \frac{d\gm(y)}{V(y,t)} \, \frac{dt}{t}\right)^{1/q} \leq \mc{A}^q f(x) + \mc{A}^q f(z) \lesssim \norm{f}_{L^q(X^+)}.
	\end{align*}
	Therefore $\mc{A}^q f$ is continuous.
\end{proof}

Having established these lemmas, we can prove the following (half-)extension of Proposition \ref{rintprop}.

\begin{prop}\label{rintprop2}
	Let all numerical parameters be as in the statement of Theorem \ref{rint2}.
	Then for all $f \in L^q_c(X^+)$ the function
	\begin{equation}\label{measurable-function}
		x \mapsto \norm{\mb{1}_{\gG(x)} f \mid (L_{s_0}^q(X^+), L_{s_1}^q(X^+))_{\gq,p_\gq}}
	\end{equation}
	is measurable on $X$ (using the discrete characterisation of the real interpolation quasi-norm), and we have
	\begin{equation}\label{eqn:rintprop2}
		\norm{f \mid (T_{s_0}^{p_0,q}, T_{s_1}^{p_1,q})_{\gq,p_\gq}} \gtrsim \norm{x \mapsto \norm{\mb{1}_{\gG(x)} f \mid (L_{s_0}^q(X^+), L_{s_1}^q(X^+))_{\gq,p_\gq}}}_{L^{p_\gq}(X)}.
	\end{equation}
	We denote the quantity on the right hand side of \eqref{eqn:rintprop} by $\norm{f \mid I^{p_\gq, q}_{s_0,s_1,\gq}}$. 
\end{prop}

\begin{proof}
First we take care of measurability.
Using Lemma \ref{Klem}, for $x \in X$ we write
\begin{align*}
	&\norm{\mb{1}_{\gG(x)} f \mid (L_{s_0}^q(X^+), L_{s_1}^q(X^+))_{\gq,p_\gq}}^{p_\gq} \\
	&= \sum_{k \in \ZZ} 2^{-kp_\gq \gq} K\left(2^k, \mb{1}_{\gG(x)} f ; L_{s_0}^q(X^+), L_{s_1}^q(X^+)\right)^{p_\gq} \\
	&= \sum_{k \in \ZZ} 2^{-kp_\gq \gq} \inf_{f = \gf_0 + \gf_1} \left( \mc{A}^q(V^{-s_0}\gf_0)(x) + 2^k \mc{A}^q(V^{-s_1}\gf_1)(x)\right)^{p_\gq}
\end{align*}
where the infima are taken over all decompositions $f = \gf_0 + \gf_1$ in $L^0(X^+)$ with $\gf_0 \in L_{s_0}^q(X^+)$ and $\gf_1 \in L_{s_1}^q(X^+)$ cylindrically supported.
By Lemma \ref{A-continuity}, we have that $\mc{A}^q(V^{-s_0}\gf_0)$ and $\mc{A}^q(V^{-s_1}\gf_1)$ are continuous.
Hence for each $k \in \ZZ$ and for every such decomposition $f = \gf_0 + \gf_1$ the function $x \mapsto \mc{A}^q(V^{-s_0} \gf_0)(x) + 2^k \mc{A}^q(V^{-s_1} \gf_1)(x)$ is continuous.
The infimum of these functions is then upper semicontinuous, therefore measurable.

Next, before beginning the proof of the estimate \eqref{eqn:rintprop2}, we apply the power theorem with $A_0 = T_{s_0}^{p_0,q}$, $A_1 = T_{s_1}^{p_1,q}$, $\gr_0 = p_0$, $\gr_1 = p_1$, and $\gs = p_\gq$.
Then we have $\gr = p_\gq$, $\gh = \gq p_\gq / p_1$, $r = 1$, and the relation $p_\gq = (1-\gh)p_0 + \gh p_1$ is satisfied.
We conclude that
\begin{equation*}
	((T^{p_0,q}_{s_0}, T^{p_1,q}_{s_1})_{\gq,p_\gq})^{p_\gq} \simeq ((T_{s_0}^{p_0,q})^{p_0}, (T_{s_1}^{p_1,q})^{p_1})_{\gq p_\gq / p_1,1}.
\end{equation*}
Thus it suffices for us to prove
\begin{equation}\label{powerreduction}
	\norm{f \mid ((T_{s_0}^{p_0,q})^{p_0}, (T_{s_1}^{p_1,q})^{p_1})_{\gq p_\gq / p_1,1}} \gtrsim \norm{f \mid I^{p_\gq,q}_{s_0,s_1,\gq}}^{p_\gq}
	\end{equation}
for all $f \in L_c^q(X^+)$.

We write
\begin{align}
	&\norm{f \mid ((T_{s_0}^{p_0,q})^{p_0},(T_{s_1}^{p_1,q})^{p_1})_{\gq p_\gq / p_1, 1}} \nonumber\\
	&= \sum_{k \in \ZZ} 2^{-k\gq p_\gq / p_1} K\left( 2^k, f ; (T_{s_0}^{p_0,q})^{p_0}, (T_{s_1}^{p_1,q})^{p_1} \right) \nonumber\\
	&= \sum_{k \in \ZZ} 2^{-k\gq p_\gq / p_1} \inf_{f = \gf_0 + \gf_1} \left( \norm{\gf_0}_{T^{p_0,q}_{s_0}}^{p_0} + 2^k \norm{\gf_1}_{T_{s_1}^{p_1,q}}^{p_1} \right) \nonumber\\
	&= \sum_{k \in \ZZ} 2^{-k\gq p_\gq / p_1} \inf_{f = \gf_0 + \gf_1} \int_X \mc{A}^q (V^{-s_0} \gf_0)(x)^{p_0} + 2^k \mc{A}^q (V^{-s_1} \gf_1)(x)^{p_1} \, d\gm(x) \nonumber\\
	&\geq \sum_{k \in \ZZ} 2^{-k\gq p_\gq / p_1} \int_X \inf_{f = \gf_0 + \gf_1} \left( \mc{A}^q(V^{-s_0} \gf_0)(x)^{p_0} + 2^k \mc{A}^q (V^{-s_1} \gf_1)(x)^{p_1} \right) \, d\gm(x) \nonumber\\
	&= \sum_{k \in \ZZ} 2^{-k\gq p_\gq / p_1} \int_X K\left(2^k, \mb{1}_{\gG(x)} f(x) ; L_{s_0}^q(X^+)^{p_0}, L_{s_1}^q(X^+)^{p_1}\right) \, d\gm(x) \label{lemline}\\
	&= \int_X \norm{\mb{1}_{\gG(x)} f \mid (L_{s_0}^q(X^+)^{p_0}, L_{s_1}^q(X^+)^{p_1})_{\gq p_\gq / p_1, 1}} \, d\gm(x) \nonumber\\
	&\simeq \int_X \norm{\mb{1}_{\gG(x)} f \mid (L_{s_0}^q(X^+), L_{s_1}^q(X^+))_{\gq, p_\gq}}^{p_\gq} \, d\gm(x) \label{powline}\\
	&= \norm{f \mid I_{s_0,s_1,\gq}^{p_\gq,q}}^{p_\gq} \nonumber
\end{align}
where again the infima are taken over cylindrically supported $\gf_0$ and $\gf_1$.
The equality \eqref{lemline} is due to Lemma \ref{Klem}.
The equivalence \eqref{powline} follows from the power theorem.
This completes the proof of Proposition \ref{rintprop2}.
\end{proof}

As a corollary, we obtain half of the desired interpolation result.
\begin{cor}\label{gilbcor2}
	Let all numerical parameters be as in the statement of Theorem \ref{rint2}, and suppose that $X$ is AD-regular of dimension $n$ and unbounded.
	Then
	\begin{equation}
		(T^{p_0,q}_{s_0}, T^{p_1,q}_{s_1})_{\gq,p_\gq} \hookrightarrow Z^{p_\gq,q}_{s_\gq}(X;c_0,c_1).
	\end{equation}
\end{cor}

\begin{proof}
	This follows from Theorem \ref{gilbcor}, Remark \ref{identn-rmk}, and the density of $L_c^q(X^+)$ in $(T^{p_0,q}_{s_0}, T^{p_1,q}_{s_1})_{\gq,p_\gq}$ (which follows from the fact that $L_c^q(X^+)$ is dense in both $T^{p_0,q}_{s_0}$ and $T^{p_1,q}_{s_1}$, which is due to Lemma \ref{density-completeness}).
\end{proof}

We now prove the reverse containment in the Euclidean case.
This rests on a dyadic characterisation of the spaces $Z_s^{p,q}(\RR^n;c_0,c_1)$.
A \emph{standard (open) dyadic cube} is a set $Q \subset \RR^n$ of the form
\begin{equation}\label{cube}
	Q = \prod_{i=1}^n (2^k x_i, 2^k(x_i + 1))
\end{equation}
for some $k \in \ZZ$ and $x \in \ZZ^n$.
For $Q$ of the form \eqref{cube} we set $\ell(Q) := 2^k$ (the sidelength of $Q$), and we denote the set of all standard dyadic cubes by $\mc{D}$.
For every $Q \in \mc{D}$ we define the associated \emph{Whitney cube}
\begin{equation*}
	\overline{Q} := Q \times (\ell(Q), 2\ell(Q)),
\end{equation*}
and we define $\mc{G} := \{\overline{Q} : Q \in \mc{D}\}$.
We write $\RR^{n+1}_+ := (\RR^n)^+ = \RR^n \times (0,\infty)$.
Note that $\mc{G}$ is a partition of $\RR^{n+1}_+$ up to a set of measure zero.

The following proposition is proven by a simple covering argument.
\begin{prop}\label{dyadic-characterisation}
	Let $p,q \in (0,\infty)$, $s \in \RR$, $c_0 > 0$ and $c_1 > 1$.
	Then for all $f \in L^0(\RR^{n+1}_+)$,
	\begin{equation*}
		\norm{f}_{Z_s^{p,q}(\RR^n ; c_0,c_1)} \simeq_{c_0,c_1} \left( \sum_{\overline{Q} \in \mc{G}} \ell(Q)^{n(1-ps)} [|f|^q]^{p/q}_{\overline{Q}} \right)^{1/p},
	\end{equation*}
	where
	\begin{equation*}
		[|f|^q]_{\overline{Q}} := \bariint_{\overline{Q}}^{} |f(y,t)|^q \, dy \, dt.
	\end{equation*}
\end{prop}

As a consequence, we gain a convenient embedding.

\begin{cor}\label{cor:ZTemb}
	Suppose $q \in (0,\infty)$, $p \in (0,q]$, and $s \in \RR$.
	Then
	\begin{equation*}
		Z_s^{p,q}(\RR^n) \hookrightarrow T_s^{p,q}(\RR^n).
	\end{equation*}
\end{cor}

\begin{proof}
	We have
	\begin{align}
		\norm{f}_{T_s^{p,q}(\RR^n)}
		&\simeq \left( \int_{\RR^n} \left( \iint_{\gG(x)} |t^{-ns} f(y,t)|^q \, \frac{dy \, dt}{t^{n+1}} \right)^{p/q} \, dx \right)^{1/p} \nonumber \\
		&\leq \left( \int_{\RR^n} \left( \sum_{\overline{Q} \in \mc{G}} \mb{1}_{\overline{Q} \cap \gG(x) \neq \varnothing}(\overline{Q}) \iint_{\overline{Q}} |t^{-ns} f(y,t)|^q \, \frac{dy \, dt}{t^{n+1}} \right)^{p/q} \, dx \right)^{1/p} \nonumber \\
		&\simeq \left( \int_{\RR^n} \left( \sum_{\overline{Q} \in \mc{G}} \mb{1}_{\overline{Q} \cap \gG(x) \neq \varnothing}(\overline{Q}) \ell(Q)^{-nsq} [ |f|^q ]_{\overline{Q}} \right)^{p/q} \, dx \right)^{1/p} \nonumber \\ 
		&\leq \left( \int_{\RR^n} \sum_{\overline{Q} \in \mc{G}} \mb{1}_{\overline{Q} \cap \gG(x) \neq \varnothing} (\overline{Q}) \ell(Q)^{-nps} [|f|^q]_{\overline{Q}}^{p/q} \, dx \right)^{1/p} \label{line:p} \\
		&= \left( \sum_{\overline{Q} \in \mc{G}} \ell(Q)^{-nps} [|f|^q]_{\overline{Q}}^{p/q} |\{ x \in \RR^n : \gG(x) \cap \overline{Q} \neq \varnothing\}| \right)^{1/p} \nonumber \\
		&\lesssim \left( \sum_{\overline{Q} \in \mc{G}} \ell(Q)^{n(1-ps)}  [|f|^q]_{\overline{Q}}^{p/q} \right)^{1/p} \label{line:cube} \\
		&\simeq \norm{f}_{Z_s^{p,q}(X;c_0,c_1)} \nonumber,
	\end{align}
	where \eqref{line:p} follows from $p/q \leq 1$, \eqref{line:cube} follows from
	\begin{equation*}
		|\{x \in \RR^n : \gG(x) \cap \overline{Q} \neq \varnothing\}| = |B(Q,2\ell(Q))| \lesssim |Q| \simeq \ell(Q)^n,
	\end{equation*}
	and the last line follows from Proposition \ref{dyadic-characterisation}.
	This proves the claimed embedding.
\end{proof}

It has already been shown by Barton and Mayboroda that the $Z$-spaces form a real interpolation scale \cite[Theorem 4.13]{BM16}, in the following sense.
We will stop referring to the parameters $c_0$ and $c_1$, as Proposition \ref{dyadic-characterisation} implies that the associated quasi-norms are equivalent.

\begin{prop}\label{Z-rint}
	Suppose that all numerical parameters are as in the statement of Theorem \ref{rint2}. Then we have the identification
	\begin{equation*}
		(Z_{s_0}^{p_0,q}(\RR^n),Z_{s_1}^{p_1,q}(\RR^n))_{\gq,p_\gq} = Z_{s_\gq}^{p_\gq,q}(\RR^n).
	\end{equation*}
\end{prop}

Now we know enough to complete the proof of Theorem \ref{rint2}.

\begin{proof}[Proof of Theorem \ref{rint2}.]
	First suppose that $p_0,p_1 \in (0,2]$.
	By Corollary \ref{cor:ZTemb} we have
	\begin{equation*}
		Z_{s_j}^{p_j,q}(\RR^n) \hookrightarrow T_{s_j}^{p_j,q}(\RR^n),
	\end{equation*}
	for $j=0,1$, and so
	\begin{equation*}
		(Z_{s_0}^{p_0,q}(\RR^n), Z_{s_1}^{p_1,q}(\RR^n))_{\gq,p_\gq} \hookrightarrow (T_{s_0}^{p_0,q}(\RR^n), T_{s_1}^{p_1,q}(\RR^n))_{\gq,p_\gq}.
	\end{equation*}
	Therefore by Proposition \ref{Z-rint} we have
	\begin{equation*}
		Z_{s_\gq}^{p_\gq,q}(\RR^n) \hookrightarrow (T_{s_0}^{p_0,q}(\RR^n), T_{s_1}^{p_1,q}(\RR^n))_{\gq,p_\gq},
	\end{equation*}
	and Corollary \ref{gilbcor2} then implies that we in fact have equality,
	\begin{equation*}
		Z_{s_\gq}^{p_\gq,q}(\RR^n) = (T_{s_0}^{p_0,q}(\RR^n), T_{s_1}^{p_1,q}(\RR^n))_{\gq,p_\gq}.
	\end{equation*}
	This equality also holds for $p_0,p_1 \in (1,\infty)$ by Theorem \ref{rint}.
	By reiteration, this equality holds for all $p_0,p_1 \in (0,\infty)$.
	The proof of Theorem \ref{rint2} is now complete.
\end{proof}

\begin{rmk}
	This can be extended to general unbounded AD-regular spaces by establishing a dyadic characterisation along the lines of Proposition \ref{dyadic-characterisation} (replacing Euclidean dyadic cubes with a more general system of `dyadic cubes'), and then proving analogues of Corollary \ref{cor:ZTemb} and Proposition \ref{Z-rint} using the dyadic characterisation.
	The Euclidean applications are enough for our planned applications, and the Euclidean argument already contains the key ideas, so we leave further details to any curious readers.
\end{rmk}

\subsection{Hardy--Littlewood--Sobolev embeddings}

In this section we prove the following embedding theorem.

\begin{thm}[Weighted tent space embeddings]\label{embthm1}
	Suppose $X$ is doubling.
	Let $0 < p_0 < p_1 \leq \infty$, $q \in (0,\infty]$ and $s_0 > s_1 \in \RR$.
	Then we have the continuous embedding
	\begin{equation*}
		T^{p_0,q}_{s_0} \hookrightarrow T^{p_1,q}_{s_1}
	\end{equation*}
	whenever $s_1 - s_0 = \gd_{p_0,p_1}$.
	Furthermore, when $p_0 \in (0,\infty]$, $q \in (1,\infty)$, and $\ga > 0$, we have the embedding
	\begin{equation*}
		T^{p_0,q}_{s_0} \hookrightarrow T^{\infty,q}_{s_1;\ga}
	\end{equation*}
	whenever $(s_1 + \ga) - s_0 = \gd_{p_0,\infty}$.
\end{thm}

	These embeddings can be thought of as being of Hardy--Littlewood--Sobolev-type, in analogy with the classical Hardy--Littlewood--Sobolev embeddings of homogeneous Triebel--Lizorkin spaces (see for example \cite[Theorem 2.1]{bJ77}).
	
The proof of Theorem \ref{embthm1} relies on the following atomic estimate.
Note that no geometric assumptions are needed here.

\begin{lem}\label{atomic-coincidences}
	Let $1 \leq p \leq q \leq \infty$ and $s_0 > s_1 \in \RR$ with $s_1 - s_0 = \gd_{1,p}$.
	Suppose that $a$ is a $T^{1,q}_{s_0}$-atom.
	Then $a$ is in $T^{p,q}_{s_1}$, with $\norm{a}_{T^{p,q}_{s_1}} \leq 1$.
\end{lem}

\begin{proof}
		Suppose that the atom $a$ is associated with the ball $B \subset X$.
		When $p \neq \infty$, using the fact that $B(x,t) \subset B$ whenever $(x,t) \in T(B)$ and that $-\gd_{1,p} > 0$, we have
		\begin{align*}
			\norm{a}_{T^{p,q}_{s_1}}
			&= \norm{\mc{A}^q(V^{-s_1} a)}_{L^p(B)} \\
			&\leq \norm{V^{-\gd_{1,p}}}_{L^\infty(T(B))} \norm{\mc{A}^q(V^{-s_0} a)}_{L^p(B)}\\
			&\leq \gm(B)^{\gd_{p,1}} \gm(B)^{\gd_{q,p}} \norm{a}_{T^{q,q}_{s_0}} \\
			&\leq \gm(B)^{\gd_{p,1} + \gd_{q,p} + \gd_{1,q}} \\
			&= 1,
		\end{align*}
		where we used H\"older's inequality with exponent $q/p \geq 1$ in the third line.
		
		When $p = q = \infty$ the argument is simpler: we have
		\begin{align*}
			\norm{a}_{T^{\infty,\infty}_{s_1}}
			&= \norm{V^{-s_0-\gd_{1,\infty}} a}_{L^\infty(T(B))} \\
			&\leq \norm{V^{-\gd_{1,\infty}}}_{L^\infty(T(B))} \norm{V^{-s_0} a}_{L^\infty(T(B))} \\
			&\leq \gm(B)^{\gd_{\infty,1}} \gm(B)^{\gd_{1,\infty}} \\
			&= 1
		\end{align*}
		using the same arguments as before (without needing H\"older's inequality).
\end{proof}

Now we will prove the embedding theorem.
Here is a quick outline of the proof.
First we establish the first statement for $p_0 = 1$ and $1 < p_1 \leq q$ by using part (1) of Lemma \ref{atomic-coincidences}.
A convexity argument extends this to $0 < p_0 < p_1 \leq q$, with $q > 1$.
Duality then gives the case $1 < q \leq p_0 < p_1 \leq \infty$, including when $p_1 = \infty$ and $\ga \neq 0$.
A composition argument completes the proof with $q > 1$.
Finally, we use another convexity argument to allow for $q \in (0,1]$ (with $p_1 < \infty$).
To handle the second statement, we argue by duality again.

\begin{proof}[Proof of Theorem \ref{embthm1}]
	The proof is split into six steps, corresponding to those of the outline above.
	
	\textbf{Step 1.}
	First suppose that $f \in T^{1,q}_{s_0}$ and $1 \leq p_1 \leq q$.	By the weighted atomic decomposition theorem, we can write $f = \sum_k \gl_k a_k$ where each $a_k$ is a $T^{1,q}_{s_0}$-atom, with the sum converging in $T^{1,q}_{s_0}$.
	By Lemma \ref{atomic-coincidences} we have
	\begin{equation*}
		\norm{f}_{T^{p_1,q}_{s_1}} \leq \norm{\gl_k}_{\ell^1(\NN)}.
	\end{equation*}
	Taking the infimum over all atomic decompositions yields the continuous embedding
	\begin{equation}\label{emb1}
		T^{1,q}_{s_0} \hookrightarrow T^{p_1,q}_{s_1} \quad (1 < p_1 \leq q \leq \infty, \quad s_1 - s_0 = \gd_{1,p_1}).
	\end{equation}
	
	\textbf{Step 2.}
	Now suppose $0 < p_0 < p_1 \leq q$, $s_1 - s_0 = \gd_{p_0,p_1}$, and $f \in T^{p_0,q}_{s_0}$.
	Using \eqref{emb1} and noting that $q/p_0 > 1$ and
	\begin{equation*}
		p_0 s_1 - p_0 s_0 = p_0 \gd_{p_0,p_1} = \gd_{1, p_1/p_0},
	\end{equation*}
	we have
	\begin{align*}
		\norm{f}_{T^{p_1,q}_{s_1}}
		&= \norm{f^{p_0} \mid T^{p_1/p_0,q/p_0}_{p_0 s_1}}^{1/p_0} \\
		&\lesssim \norm{f^{p_0} \mid T^{1,q/p_0}_{p_0 s_0}}^{1/p_0} \\
		&= \norm{f}_{T^{p_0,q}_{s_0}},
	\end{align*}
	which yields the continuous embedding
	\begin{equation}\label{emb2}
		T^{p_0,q}_{s_0} \hookrightarrow T^{p_1,q}_{s_1} \quad (0 < p_0 < p_1 \leq q \leq \infty, \quad q > 1, \quad s_1 - s_0 = \gd_{p_0,p_1}).
	\end{equation}
	
	\textbf{Step 3.}
	We now use a duality argument.
	Suppose $1 < q \leq p_0 < p_1 \leq \infty$.
	Define $\gp_0 := p_1^\prime$, $\gp_1 := p_0^\prime$, $\gr := q^\prime$, $\gs_0 := -s_1$, and $\gs_1 := -s_0$, with $s_1 - s_0 = \gd_{p_0,p_1}$.
	Then
	\begin{equation*}
		\gs_1 - \gs_0 = -s_0 + s_1 = \gd_{p_0,p_1} = \gd_{\gp_0,\gp_1},
	\end{equation*}
	and so \eqref{emb2} gives the continuous embedding
	\begin{equation*}
		T^{\gp_0,\gr}_{\gs_0} \hookrightarrow T^{\gp_1,\gr}_{\gs_1}.
	\end{equation*}
	Taking duals and considering that $T^{\gp_0,\gr_0}_{\gs_0}$ is dense in $T_{\gs_1}^{\gp_1,\gr}$ results in the continuous embedding
	\begin{equation}\label{emb3}
		T^{p_0,q}_{s_0} \hookrightarrow T^{p_1,q}_{s_1} \quad (1 < q \leq p_0 < p_1 \leq \infty, \quad s_1 - s_0 = \gd_{p_0,p_1}).
	\end{equation}
	
	\textbf{Step 4.}
	Now suppose that $0 < p_0 \leq q \leq p_1 \leq \infty$ and $q > 1$, again with $s_1 - s_0 = \gd_{p_0,p_1}$.
	Then combining \eqref{emb2} and \eqref{emb3} gives continuous embeddings
	\begin{equation}\label{emb4}
		T^{p_0,q}_{s_0} \hookrightarrow T^{q,q}_{s_0 + \gd_{p_0,q}} \hookrightarrow T^{p_1,q}_{s_0 + \gd_{p_0,q} + \gd_{q,p_1}} = T^{p_1,q}_{s_1}.
	\end{equation}
	
	\textbf{Step 5.} Finally, suppose $q \leq 1$, and choose $M > 0$ such that $q/M > 1$.
	Then using a similar argument to that of Step 2, with $Ms_1 - Ms_0 = M\gd_{p_0,p_1} = \gd_{p_0/M,p_1/M}$,
	\begin{align*}
		\norm{f}_{T^{p_1,q}_{s_1}}
		&= \norm{ f^M \mid T^{p_1/M, q/M}_{Ms_1}}^{1/M} \\
		&\lesssim \norm{f^M \mid T^{p_0/M, q/M}_{Ms_0}}^{1/M} \\
		&= \norm{f}_{T^{p_0,q}_{s}}.
	\end{align*}
	
	All possible positions of $q$ relative to $0 < p_0 < p_1 \leq \infty$ have thus been covered, so the proof of the first statement is complete.
	
	\textbf{Step 6.}
	For the second statement, we let $(s_1 + \ga) - s_0 = \gd_{p_0,\infty}$, and first we suppose that $p_0 \in (1,\infty]$.
	Let
	\begin{align*}
		\gp_0 &:= (1+\ga)^{-1} \in (0,1), \\
		\gp_1 &:= p_0^\prime \in (1,\infty], \\
		\gr &= q^\prime, \quad \gs_0 = -s_1, \quad \gs_1 = -s_0.
	\end{align*}
	Then $\ga = \gd_{1,\gp_0} = \gd_{p_1,\infty}$ and so we have
	\begin{equation*}
		\gs_1 - \gs_0 = \gd_{p_0,\infty} - \ga = \gd_{1,\gp_1} - \gd_{1,\gp_0} = \gd_{\gp_0,\gp_1},
	\end{equation*}
	which yields
	\begin{equation*}
		T^{\gp_0,\gr}_{\gs_0} \hookrightarrow T^{\gp_1,\gr}_{\gs_1}.
	\end{equation*}
	Taking duals yields
	\begin{equation*}
		T^{p_0,q}_{s_0} \hookrightarrow T^{\infty,q}_{s_1,\ga},
	\end{equation*}
	which completes the proof when $p_0 \in (1,\infty]$.
	One last convex reduction argument, as in Step 2, completes the proof.
\end{proof}

	We remark that this technique also yields the embedding $T^{\infty,q}_{s_0,\ga_0} \hookrightarrow T^{\infty,q}_{s_1,\ga_1}$ when $(s_1 + \ga_1) - (s_0 + \ga_0) = 0$, $s_0 > s_1$, and $0 \leq \ga_0 < \ga_1$.
		
	\begin{rmk}
	The embeddings of Theorems \ref{embthm1}, at least for $p,q \in (1,\infty)$, also hold with $Z^{p,q}_s$ replacing $T^{p,q}_s$ on either side (or both sides) of the embedding.
	This can be proven by writing $Z^{p,q}_s$ as a real interpolation space between tent spaces $T^{\td{p},q}_{\td{s}}$ with $\td{p}$ near $p$ and $\td{s}$ near $s$, applying the tent space embedding theorems, and then interpolating again.
	These embeddings can also be proven `by hand', even for $p, q \leq 1$.
	We leave the details to any curious readers.
	\end{rmk}
	
\section{Deferred proofs}

\subsection{$T^{p,\infty}$-$L^\infty$ estimates for cylindrically supported functions}\label{defprf:duality}

The following lemma, which extends \cite[Lemma 3.3]{aA14} to the case $q = \infty$, is used in the proof that $T^{p,\infty}$ is complete (see Proposition \ref{density-completeness}).

\begin{lem}\label{loclem}
	Suppose that $X$ is doubling and let $K \subset X^+$ be cylindrical.
	Then for all $p \in (0,\infty]$ and all $f \in L^0(X^+)$,
	\begin{equation*}
		\norm{\mb{1}_K f}_{T^{p,\infty}} \lesssim_{K,p} \norm{f}_{L^\infty(K)} \lesssim_{K,p} \norm{f}_{T^{p,\infty}}.
	\end{equation*}
\end{lem}

\begin{proof}
	When $p = \infty$ this reduces to
	\begin{equation*}
		\norm{\mb{1}_K f}_{L^\infty(X^+)} = \norm{f}_{L^\infty(K)} \leq \norm{f}_{L^\infty(X^+)},
	\end{equation*}
	which is immediate.
	Thus it suffices to prove the result for $p < \infty$.
	Write $K \subset B_K \times (\gk_0, \gk_1)$ for some ball $B_K = B(c_K,r_K) \subset X$ and $0 < \gk_0 < \gk_1 < \infty$.
	
	To prove that $\norm{\mb{1}_K f}_{T^{p,\infty}} \lesssim_{K,p} \norm{f}_{L^\infty(K)}$, observe that
	\begin{align*}	
		\norm{\mb{1}_K f}_{T^{p,\infty}} &\leq \norm{f}_{L^\infty(K)} \gm\{x \in X : \gG(x) \cap K \neq \varnothing\}^{1/p} \\
		&\leq \norm{f}_{L^\infty(K)} V(c_K, r_K + \gk_1)^{1/p}
	\end{align*}
	because if $x \notin B(c_K,r_K + \gk_1)$ then $\gG(x) \cap (B_K \times (\gk_0, \gk_1)) = \varnothing$.
	Note also that $V(c_K, r_K + \gk_1)$ is finite and depends only on $K$.

	Now we will prove that $\norm{f}_{L^\infty(K)} \lesssim_{K,p} \norm{f}_{T^{p,\infty}}$.
	First note that the doubling property implies that for all $R > 0$ and for all balls $B \subset X$,
	\begin{equation}\label{inftyclaim}
		\inf_{x \in B} \gm(B(x,R)) \gtrsim_{X,R,r_B} \gm(B). 
	\end{equation}
	Indeed, if $x \in B$ and $R \leq 2r_B$ then
	\begin{equation*}
		\gm(B) \leq \gm(B(x,R(2r_B R^{-1}))) \lesssim_X (2r_B R^{-1})^n \gm(B(x,R)).
	\end{equation*}
	where $n \geq 0$ is the doubling dimension of $X$.
	If $R > 2r_B$ then since $2r_B R^{-1} < 1$, we have $\gm(B) \leq \gm(B(x,R))$.
	
	Let $(x_j)_{j \in \NN}$ be a countable dense subset of $B_K$.
	Then we have
	\begin{equation*}
		K = \bigcup_{j \in \NN} (\gG(x_j) + \gk_0) \cap K.
	\end{equation*}
	By definition the set $\{(y,t) \in K : |f(y,t)| > 2^{-1}\norm{f}_{L^\infty(K)}\}$ has positive measure, so there exists $j \in \NN$ such that $|f(y,t)| > 2^{-1} \norm{f}_{L^\infty(K)}$ for $(y,t)$ in some subset of $(\gG(x_j) + \gk_0) \cap K$ with positive measure.
	Since $(\gG(x_j) + \gk_0) \cap K \subset \gG(x) \cap K$ for all $x \in B(x_j,\gk_0)$, we have that $\mc{A}^\infty(f)(x) \geq 2^{-1}\norm{f}_{L^\infty(K)}$ for all $x \in B(x_j,\gk_0)$.
	Therefore, using \eqref{inftyclaim},
	\begin{align*}
		\norm{\mc{A}^\infty f}_{L^p(X)}
		&\geq \frac{1}{2} \gm(B(x_j,\gk_0))^{1/p} \norm{f}_{L^\infty(K)} \\
		&\gtrsim_{X,K} \gm(B_K)^{1/p} \norm{f}_{L^\infty(K)} \\
		&\simeq_{K,p} \norm{f}_{L^\infty(K)}.
	\end{align*}
	This completes the proof of the lemma.
\end{proof}

\subsection{$T^{p,\infty}$ atomic decomposition}\label{defprf:atomic}

As stated above, the atomic decomposition theorem for $T^{p,\infty}$ can be proven by combining the arguments of Coifman--Meyer--Stein (who prove the result in the Euclidean case) and Russ (who proves the atomic decomposition of $T^{p,2}(X)$ for $0 < p \leq 1$ when $X$ is doubling).

First we recall a classical lemma (see for example \cite[Lemma 2.2]{eR07}), which combines a Vitali-type covering lemma with a partition of unity.
This is proven by combining the Vitali-type covering of Coifmann--Weiss \cite[Th\'eor\`eme 1.3]{CW71} with the partition of unity of Mac{\'i}as--Segovia \cite[Lemma 2.16]{MS79}. 

\begin{lem}\label{lem:russ}
	Suppose that $X$ is doubling, and let $O$ be a proper open subset of $X$ of finite measure.
	For all $x \in X$ write $r(x) := \dist(x,O^c)/10$.
	Then there exists $M > 0$, a countable indexing set $I$, and a collection of points $\{x_i\}_{i \in I}$ such that
	\begin{itemize}
		\item $O = \cup_{i \in I} B(x_i, r(x_i))$,
		\item if $i,j \in I$ are not equal, then $B(x_i, r(x_i)/4)$ and $B(x_j, r(x_j)/4)$ are disjoint, and
		\item for all $i \in I$, there exist at most $M$ indices $j \in I$ such that $B(x_j,5r(x_j))$ meets $B(x_i,5r(x_i))$.
	\end{itemize}
	Moreover, there exist a collection of measurable functions $\{\map{\gf_i}{X}{[0,1]}\}_{i \in I}$ such that
	\begin{itemize}
		\item $\supp \gf_i \subset B(x_i,2r(x_i))$,
		\item $\sum_i \gf_i = \mb{1}_O$ (for each $x \in X$ the sum $\sum_i \gf_i(x)$ is finite due to the third condition above).
	\end{itemize}
	\end{lem}

Now we can follow a simplified version of the argument of Russ, which is essentially the argument of Coifman--Meyer--Stein with the partition of unity of Lemma \ref{lem:russ} replacing the use of the Whitney decomposition.

\begin{proof}[Proof of Theorem \ref{atomic-decomposition}, with $q = \infty$]
	Suppose $f \in T^{p,\infty}$, and for each $k \in \ZZ$ define the set
	\begin{equation*}
		O^k := \{x \in X : \mc{A}^\infty f (x) > 2^k\}.
	\end{equation*}
	The sets $O^k$ are open by lower semicontinuity of $\mc{A}^\infty f$ (Lemma \ref{ACsemicontinuity}), and the function $f$ is essentially supported in $\cup_{k \in \ZZ} T(O^k) \sm T(O^{k+1})$.
	Thus we can write
	\begin{equation}\label{eqn:atomic-1}
		f = \sum_{k \in \ZZ} \mb{1}_{T(O^k) \sm T(O^{k+1})} f.
	\end{equation}
	
	\textbf{Case 1: $\gm(X) = \infty$.}
	In this case we must have $\gm(O^k) < \infty$ for each $k \in \ZZ$, for otherwise we would have $\norm{\mc{A}^\infty f}_{L^p(X)} = \infty$ and thus $f \notin T^{p,\infty}$.
	Hence for each $k \in \ZZ$ there exist countable collections of points $\{x_i^k\}_{i \in I^k} \subset O^k$ and measurable functions $\{\gf_i^k\}_{i \in I^k}$ as in Lemma \ref{lem:russ}.
	Combining \eqref{eqn:atomic-1} with $\sum_{i \in I^k} \gf_i^k = \mb{1}_{O^k}$ and $T(O^k) \subset O^k \times \RR_+$, we can write
	\begin{align*}
		f(y,t) &= \sum_{k \in \ZZ} \sum_{i \in I^k} \gf_i^k(y) \mb{1}_{T(O^k) \sm T(O^{k+1})}(y,t) f(y,t) \\
			&= \sum_{k \in \ZZ} \sum_{i \in I^k} \td{a}_i^k(y,t).
	\end{align*}
	Note that
	\begin{equation}\label{eqn:easy-atom-est}
		\norm{\td{a}_i^k}_{L^\infty(X^+)} \leq \esssup_{(y,t) \notin T(O^{k+1})} |f(y,t)| \leq 2^{k+1};
	\end{equation}
	we shall pause to prove the second inequality.
	Since $X^+$ is a Lindel\"of space, there exists a countable set $J$ such that $T(O^{k+1})^c = \cup_{i \in J} \gG(x_i)$, where $x_i \in (O^{k+1})^c$ for each $i \in J$.
	For each $i \in J$ choose a null subset $N_i \subset \gG(x_i)$ with $|f(y,t)| \leq \mc{A}^\infty f(x_i)$ for all $(y,t) \in \gG(x_i) \sm N_i$.
	Then $|f(y,t)| \leq 2^{k+1}$ on $T(O^{k+1}) \sm \cup_{i \in J} N_i$, which proves the second inequality of \eqref{eqn:easy-atom-est}.\footnote{I thank the referee specifically for pointing out the need for this argument.}
	
	Define
	\begin{equation*}
		a_i^k := 2^{-(k+1)} \gm(B_i^k)^{-1/p} \td{a}_i^k,
	\end{equation*}
	where $B_i^k := B(x_i^k, 14r(x_i^k))$.
	We claim that $a_i^k$ is a $T^{p,\infty}$-atom associated with the ball $B_i^k$.
	The estimate \eqref{eqn:easy-atom-est} immediately implies the size condition
	\begin{equation*}
		\norm{a_i^k}_{T^{\infty,\infty}} \leq \gm(B_i^k)^{\gd_{p,\infty}},
	\end{equation*}
	so we need only show that $a_i^k$ is essentially supported in $T(B_i^k)$.
	To show this, it is sufficient to show that if $y \in B(x_i^k, 2r(x_i^k))$ and $d(y, (O^k)^c) \geq t$, then $d(y,(B_i^k)^c) \geq t$.
	Suppose $z \notin B_i^k$ (such a point exists because $\gm(B_i^k) < \gm(X) = \infty$), $\ge > 0$ and $u \notin O^k$ such that
	\begin{equation*}
		d(x_i^k, u) < d(x_i^k, (O^k)^c) + \ge = 10r(x_i^k) + \ge.
	\end{equation*}
	Then we have
	\begin{align*}
		d(y,z) + \ge &\geq d(z,x_i^k) - d(x_i^k, y) + \ge\\
		&\geq 12 r(x_i^k) + \ge \\
		&= 2r(x_i^k) + 10r(x_i^k) + \ge \\
		&> d(y,x_i^k) + d(x_i^k, u) \\
		&\geq d(y,u) \\
		&\geq t,
	\end{align*}
	where the last line follows from $u \notin O^k$ and $d(y,(O^k)^c) \geq t$.
	Since $z \notin B_i^k$ and $\ge > 0$ were arbitrary, this shows that $d(y,(B_i^k)^c) \geq t$ as required, which proves that $a_i^k$ is a $T^{p,\infty}$-atom associated with $B_i^k$.
	
	Thus we have
	\begin{equation*}
		f(y,t) = \sum_{k \in \ZZ} \sum_{i \in I^k} \gl_i^k a_i^k,
	\end{equation*}
	where
	\begin{equation*}
		\gl_i^k = 2^{k+1} \gm(B_i^k)^{1/p}.
	\end{equation*}
	It only remains to show that
	\begin{equation*}
		\sum_{k \in \ZZ} \sum_{i \in I^k} |\gl_i^k|^p \lesssim \norm{f}_{T^{p,\infty}}^p.
	\end{equation*}
	We estimate
	\begin{align}
		\sum_{k \in \ZZ} \sum_{i \in I^k} |\gl_i^k|^p
		&= \sum_{k \in \ZZ} 2^{(k+1)p} \sum_{i \in I^k} \gm(B_i^k) \nonumber\\
		&\lesssim_X \sum_{k \in \ZZ} 2^{(k+1)p} \sum_{i \in I^k} \gm(B(x_i^k, r(x_i^k)/4)) \label{line:doublinguse} \\
		&\leq \sum_{k \in \ZZ} 2^{(k+1)p} \gm(O^k) \label{line:disjointnessuse} \\
		&\lesssim_p p\sum_{k \in \ZZ} \int_{2^{k-1}}^{2^k} t^{p-1} \gm(\{x \in X : \mc{A}^\infty f(x) > t\}) \, dt \nonumber \\
		&= \norm{\mc{A}^\infty f}_{L^p(X)}^p \nonumber \\
		&= \norm{f}_{T^{p,\infty}}, \nonumber
	\end{align}
	using doubling in \eqref{line:doublinguse} and pairwise disjointness of the balls $B(x_i^k, r(x_i^k)/4)$ in \eqref{line:disjointnessuse}.
	This completes the proof in the case that $\gm(X) = \infty$.
	
	\textbf{Case 2: $\gm(X) < \infty$.}
	In this case we may have $O^k = X$ for some $k \in \ZZ$, so we cannot apply Lemma \ref{lem:russ} as before.
	One can follow the argument of Russ \cite[page 131]{eR07}, which shows that the partition of unity is not required for such $k$.
	With this modification, the argument of the previous case still works.
	We omit the details.
\end{proof}

\footnotesize
\bibliographystyle{amsplain}
\bibliography{WTS-bib}  

\end{document}